\documentclass[a4paper,11pt,dvipdfm]{article}
\usepackage{amsmath,amsfonts,amsthm,amstext,amssymb,amscd,amsbsy,bezier,mathrsfs}
\usepackage{stmaryrd} 
\usepackage[latin1]{inputenc}     
\usepackage[pdftex]{color,graphicx} 
\usepackage[all]{xy}
\usepackage[top=3cm,left=3cm,right=3cm,bottom=2cm]{
geometry} 

\newcommand{\stab}{\mbox{\normalfont Stab}}
\newcommand{\Out}{\mbox{\normalfont Out}}
\newcommand{\Z}{\mathbb{Z}}

\newtheorem{theorem}{Theorem}
\newtheorem{proposition}{Proposition}
\newtheorem{definition}{Definition}
\newtheorem{lemma}{Lemma}
\newtheorem{corollary}{Corollary}
\newtheorem{conjecture}{Conjecture}

\begin{document}

\title{Code loops: automorphisms and  representations}

\author{Alexandre Grichkov and Rosemary M. Pires}

\maketitle
\begin{abstract}
In this work we construct free Moufang loop in the variety generated by code loops. We apply this construction for study the code loops.  
 Moreover, we define and determine all  basic representations of code loops of rank $3$ and $4$.
\end{abstract}

\section{Introduction} 
The study of code loops began with the paper \cite{2} of Griess. Let's recall the definition of code loop. First let ${\bf F}_2^n$ be a $n$-dimensional vector space over a field with two elements ${\bf F}_2=\{0,1\}$. 
For vectors $u=(u_1,\dots,u_n)$ and $v=(v_1,\dots,v_n)$ we define $|v|=|\{i |v_i=1\}|$  and  $|u \cap v|=|\{i |u_i=v_i=1\}|.$

A double even code is a subspace $V\subseteq{{\bf F}_{2}^{n}}$ such that $|v|\equiv{0\;}(mod{\;4})$ and \linebreak $|u \cap v|\equiv{0\;}(mod{\;2})$ for all $u,v \in V$.

Let $V$ be a double even code and $L(V)$ be the set $ \{1,-1\} \times V $. In \cite{2} R.Griess proved that there exists a function $\phi:V\times V\to \{1,-1\}$, called factor set,  such that 

\begin{align}
\phi(v,v)&=(-1)^{\frac{|v|}{4}}, \nonumber \\
\phi(v,w)&=(-1)^{\frac{|v{\cap}w|}{2}}\phi(w,v), \nonumber \\
\phi(0,v)&=\phi(v,0)=1,   \nonumber \\
\phi(v+w,u)&=\phi(v,w+u)\phi(v,w)\phi(w,u)(-1)^{|v{\cap}w{\cap}u|},
\end{align}
where $|v{\cap}w{\cap}u|$ denotes the number of positions in which the coordinates of $u$, $v$ and $w$ are both nonzero.

In order to define code loop, let $v,w \in V$ and $\phi(v,w) \in \{1,-1\}$, and let $(1,v)$ denotes $v$ and $(-1,v)$ denotes $-v$. We define a product  $``\cdot"$ em $L(V)$ by:
\begin{align}
&v.w = \phi(v,w)(v + w), \nonumber\\
&v.(-w) = (-v).w = -(v.w), \nonumber\\
&(-v).(-w) = v.w. 
\end{align} 

With the product defined above, Griess (1986, \cite{2}) proved that $(L(V),.)$, or merely $L(V)$, has a Moufang loop structure.
The Moufang loop $L(V)$ is called code loop. We say that $L(V)$ has rank $m$, if dim$_{{\bf F}_2}V = m$. 

Moreover, Chein and Goodaire (1990, \cite{4}) proved that  code loops have a unique nonidentity square, a unique nonidentity commutador, and a unique nonidentity associator. In other words, for any $u,v,w \in V$: 
\begin{align}
v^2 &= (-1)^{\frac{|v|}{4}}0,\nonumber \\
\left[u,v\right] &= u^{-1}v^{-1}uv = (-1)^{\frac{|u \cap v|}{2}}0, \nonumber\\
(u,v,w)&=((uv)w)( (u(vw))^{-1})=(-1)^{|u\cap v\cap w|}0. 
\end{align}

On the other hand, a Moufang loop $L$ is called \textbf{$E$-loop} if  there is a central subloop $Z$ with $2$ elements such that $L/Z \in {\cal A}$, where ${\cal A}$ is the variety of groups with identity $x^2 = 1.$ 

Chein e Goodaire (1990, \cite{4}) proved  that finite code loops may be characterized as \linebreak Moufang loops $L$ for which $|L^{2}|\leq 2$. So, directly from the proof of this result we can obtain that a finite Moufang loop $L$ is a code loop if and only if $L$ is a $E$-loop. 

In this paper, we study the representations of code loops. In Section 2, we prove that code loops of rank $n$ can be characterized as a homomorphic image of a certain free Moufang loops with $n$ generators and we introduce the concept of characteristic vectors associated with code loops. In Section 3, we present the classification of code loops of rank $3$ and $4$ and their corresponding groups of outer automorphisms. There are exactly 5 nonassociative code loops of rank $3$ (up to isomorphism) and 16 nonassociative code loops of rank $4$ (up to isomorphism). By using this, in Section 4, we determine all basic representations of code loops of rank $3$ and $4$.

\newpage
\section{Free Loops of the Variety $\cal{E}$}

In the following we introduce the variety of Moufang loops ${\cal E}$, generated by all code loops.

\begin{definition}
Let \; ${\cal E}$ be the variety of Moufang loops with the following identities:
\begin{align}
&x^4 = 1, \,\left[x,y \right]^2 =1 , \, (x,y,z)^2 = 1, \nonumber \\
&\left[x^2,y \right] = 1,\,    \left[\left[x,y\right],t \right] = 1,\, \left[(x,y,z),t \right] = 1, \nonumber \\
&(x^2,y,z) = 1 ,  \,  (\left[x,y \right],z,t) = 1, \;\; ((x,y,z),t,s)=1. \label{eq1}
\end{align}

\end{definition}

We observe that code loops are contained in ${\cal E}$. In fact, a code loop is a \linebreak Moufang loop with a Unique Nonidentity Commutator (Associator, Square). Besides, Chein and Goodaire (\cite{4},Theorem 1) proved that if $L$ is a Moufang loop with a unique nonidentity square $e$, then $e^{2}=1$ and either $L$ is an abelian group or else $[L,L]=(L,L,L)=L^{2}=\{1,e\}\subseteq {\cal Z}(L)$. Here, $[L,L]$, $(L,L,L)$, and $L^{2}$ denotes, respectively, the sets of all commutators $[x,y]$, all associators $(x,y,z)$ and all squares $x^2$, $x,y,z \in  L$. Besides, the center of $L$,  ${\cal Z}(L)$, is the set of elements of $L$ which associate with every pair of elements of $L$ and which commute with every element of $L$.

Note that those identities are not independents. It is interesting problem to find the minimal set of identities that defines the variety ${\cal E}$ (see the Conjecture below).

\begin{conjecture}\label{cn1} The variety ${\cal E}$ has the following minimal set of identities:
$$x^4 = 1, \left[x,y \right]^2 =1 ,
\left[x^2,y \right] = 1,(x^2,y,z) = 1.$$   
\end{conjecture}

Let $V$ be a ${\bf F}_2$-space with a basis $\left\{v_{1},v_{2},\dots\right\}$. We identify each element of $V$ with a finite subset of the set of natural numbers $\mathbb{N}$ as follows: We consider the biunivocal correspondence  $\sigma: V \longrightarrow {\cal P}(\mathbb{N})$, where $$v=a_{1}v_{1}+a_{2}v_{2}+\cdots \in V \longmapsto \sigma(v)=\sigma =\{i| a_{i}=~1\}\in {\cal P}(\mathbb{N}).$$

 In this way, we use the notation $\sigma \in V$ when $\sigma = \sigma (v)$, for $v \in V$. Besides, we observe that,  $\sigma(v+w)=\sigma(v)\Delta \sigma(w)$,\; where $\sigma \Delta \mu =(\sigma \setminus \mu)\cup(\mu \setminus \sigma)$.

 We consider $W=V{\wedge}V$ and $U=V{\wedge}V{\wedge}V$ antisymmetric products of $V$. This means that $W$ has a basis $\{ i\wedge j \; | \; i\wedge j = j\wedge i, i\wedge i =0; \; i,j \in \mathbb{N} \}$ and $U$ has a basis  $\{ i\wedge j\wedge k \; |\; i\wedge j\wedge k = j\wedge k\wedge i = j\wedge i\wedge k; \;i,j,k \in \mathbb{N} \}$. Moreover, $i\wedge j\wedge k = 0$ if and only if $|\{i,j,k\}| < 3$. Let $\bar{V}$ be  an isomorphic copy of $V$ with a basis $\left\{\bar{i}|i \in \mathbb{N}\right\}.$ 
 
 Let the set ${\cal F} = V.{\bar{V}}. W . U$. We want to define a product ${\cal F}{\cal F}\subseteq {\cal F}$ such that for $i,j,k \in \mathbb{N}$ we have
\begin{eqnarray}
i^{2}=\bar{i}, \;\; [i,j]=i^{-1}j^{-1}ij=i\wedge j, \;\;(i,j,k) =(ij.k)(i.jk)^{-1}= i\wedge j\wedge k . \label{eq2}
\end{eqnarray}

For the following, given an element $\sigma = \{i_{1},i_{2},\dots,i_{s}\} \in V$, we identify it in $\cal F$ with a product $\sigma = (\dots(i_{1}i_{2})i_{3})\dots)i_{s}$. Now we define
\begin{gather}
\sigma . \mu = \sigma \Delta \mu \;. \displaystyle\prod_{i \in {\sigma{\cap}\mu}}{\bar{i}} \;. \displaystyle\prod_{i \in \sigma, j \in \mu, i > j }{i\wedge j} \;. \displaystyle\prod_{i \in \sigma, j,k \in \mu, j < k }{i\wedge j\wedge k}, \;\sigma,\mu\in V.\label{eq3} \\
x^{2}=[x,y]=(x,y,z)=1, \;\;\;x\in {\bar{V}}.W.U, \;\;y,z \in \cal F. \label{eq4}
\end{gather}

By definition ${\cal Z (F)}={\bar{V}}.W.U$. Now let $v,w \in {\cal F}$ be of the form $v=v_{0}z_{0}$ and $w=w_{0}z_{1}$, where $v_{0},w_{0} \in V$ and $z_{0}, z_{1} \in {\cal Z (F)}$. Therefore, the product  of $v$ by $w$ is given by: $v.w=(v_{0}.w_{0}).z_{0}.z_{1}$.

For the following proofs we use the notation more concise, for all $\sigma, \mu, \lambda \in V$:
\begin{align}
\left[\sigma,\mu,\lambda\right]&=\displaystyle\prod_{i \in \sigma, j \in \mu, k \in \lambda , j < k }{i\wedge j\wedge k} \label{eq5}\\
\left[\sigma,\mu\right]&=\displaystyle\prod_{i \in \sigma, j \in \mu, i > j }{i\wedge j} \label{eq6}\\
\{\sigma,\mu\}&=\left\{\begin{array}{c}
\displaystyle\prod_{i \in \sigma \cap \mu }{\bar{i}}, \;\;\;\mbox{se} \;\;\;\sigma \cap \mu \neq \emptyset\\
1, \;\;\;\mbox{se} \;\;\;\;\;\;\sigma \cap \mu = \emptyset \label{eq7}
\end{array}\right.
\end{align}

\begin{lemma}\label{lemma2.3} With the above definitions, in $\cal F$ are valid the following equalities, for all $\sigma, \mu, \lambda \in V$:
\begin{gather}
\{\sigma,\mu\}\{\sigma,\lambda\}\{\sigma \Delta \mu,\lambda \Delta \sigma\}=\{\lambda,\mu\}\{\sigma,\lambda \Delta \mu\}\{\sigma \Delta \mu \Delta \lambda,\sigma\}\label{eq8}\\
\left[\sigma,\mu\right][\lambda,\sigma][\sigma \Delta \mu,\lambda \Delta \sigma]=[\mu,\lambda][\sigma,\lambda \Delta \mu][\sigma \Delta \mu \Delta \lambda,\sigma]\label{eq9}\\
\left[\lambda \Delta \tau, \mu, \sigma\right]=[\lambda, \mu, \sigma][\tau, \mu, \sigma] \nonumber\\
\left[\sigma, \lambda \Delta \tau, \mu\right]=[\sigma, \lambda, \mu][\sigma, \tau, \mu] \nonumber\\
\left[\sigma, \mu, \lambda \Delta \tau\right]=[\sigma, \mu, \lambda][\sigma, \mu, \tau] \label{eq10}\\
\left[\sigma, \lambda, \sigma\right][\sigma, \sigma, \lambda]=1 \label{eq11}\\
\left[\sigma, \lambda, \mu\right][\mu, \lambda,\sigma][\sigma,\mu, \lambda][\mu, \sigma, \lambda]=1 \label{eq12}\\
\left[\sigma \Delta \mu,\lambda \Delta \sigma, \lambda \Delta \sigma\right][\sigma,\mu,\mu][\lambda, \sigma, \sigma][\sigma \Delta \mu \Delta \lambda, \sigma, \sigma][\sigma, \mu \Delta \lambda, \mu \Delta \lambda][\mu,\lambda, \lambda]=1 \label{eq13}
\end{gather}

\end{lemma}

\begin{proof}
We denote
\begin{align}
\xi_{1}&=\sigma \cap \lambda \cap \mu , \;\;\;\xi_{2}=(\sigma \cap \mu)\setminus \lambda , \;\;\xi_{3}=(\sigma \cap \lambda)\setminus \mu , \;\;\xi_{4}=\sigma \setminus (\mu \cup \lambda),\nonumber\\
\xi_{5}&=(\lambda \cap \mu)\setminus \sigma, \;\;\xi_{6}=\mu \setminus (\sigma \cup \lambda), \;\;\xi_{7}=\lambda \setminus (\sigma \cup \mu).\nonumber
\end{align}
We write $i, ij=(ij), ijk=(ijk), \left\{i,j\right\}, \left\{ij,k\right\}, [i,j], [ij,k], [i,j,k]$ and $[ip,j,k]$\linebreak instead $\xi_{i},\; \xi_{i}\cup \xi_{j}, \xi_{i}\cup \xi_{j}\cup \xi_{k}, \left\{\xi_{i},\xi_{j}\right\}, \left\{\xi_{i} \cup \xi_{j}, \xi_{k}\right\}, [\xi_{i},\xi_{j}], [\xi_{i} \cup \xi_{j}, \xi_{k}], [\xi_{i}, \xi_{j}, \xi_{k}] $ and $[\xi_{i} \cup \xi_{p},\xi_{j}, \xi_{k}]$, respectively. 

To prove the equality (\ref{eq8}), first we see that 
\begin{equation}
\{\sigma \cup \mu, \lambda\} = \{\sigma,\lambda\}\{\mu,\lambda\},\;\;\mbox{for all} \;\sigma, \mu, \lambda \in V. \label{eq14}
\end{equation}

Thus, by the equality (\ref{eq14}) and by the fact $ \{\sigma,\mu\}=1 \;\mbox{if}\; \sigma \cap \mu = \emptyset,$ for all $\sigma, \mu \in V$, we can rewrite the right part of (\ref{eq8}) as follows: $$\{1,1\}\{5,5\}\{3,3\}\{2,2\}\{4,4\}\{1,1\}=\{5,5\}\{3,3\}\{2,2\}\{4,4\}.$$

In fact:
\begin{enumerate}
	\item \label{item1} $\{\sigma \Delta \mu \Delta \lambda,\sigma\}= \{\sigma \setminus (\mu \Delta \lambda), \sigma\}\{(\mu \Delta \lambda)\setminus \sigma, \sigma\} 
                                           =\{\xi_{4},\sigma\}\{\xi_{1},\sigma\}\{\xi_{6},\sigma\}\{\xi_{7},\sigma\}. $ 
																					
We note that $\sigma \cap \xi_{4}=\xi_{4}$, then $\{\xi_{4},\sigma\}=\{\xi_{4},\xi_{4}\}=\{4,4\}$. Analogously, $\{\xi_{1},\sigma\}=\{1,1\}$. We also note that $\xi_{6} \cap \sigma = \emptyset$ and $\xi_{7} \cap \sigma = \emptyset$, so we have $\{\xi_{6},\sigma\}=\{\xi_{7},\sigma\}=1.$ Therefore, $\{\sigma \Delta \mu \Delta \lambda,\sigma\}=\{4,4\}\{1,1\}.$

\item \label{item2} $\{\sigma, \lambda \Delta \mu\} = \{\sigma\cap (\lambda \setminus \mu), \sigma\cap (\lambda \setminus \mu)\}\{\sigma \cap (\mu \setminus \lambda),\sigma \cap (\mu \setminus \lambda)\}=\{3,3\}\{2,2\}.$
															
\item \label{item3} Since  $\lambda \cap \mu = (\lambda \cap \mu \cap \sigma)\cup((\lambda \cap \mu)\setminus \sigma)$, then
$\{\lambda, \mu\}=\{\xi_{1}, \mu\}\{\xi_{5}, \mu\}=\{1,1\}\{5,5\}.$
\end{enumerate}
Thus, by (\ref{item1}),(\ref{item2}) and (\ref{item3}), we have $$\{\lambda,\mu\}\{\sigma,\lambda \Delta \mu\}\{\sigma \Delta \mu \Delta \lambda,\sigma\}=\{5,5\}\{3,3\}\{2,2\}\{4,4\}.$$
 
Analogously, we prove that the left part of (\ref{eq8}) can be rewritten in the same way.
\begin{figure}[!hh]
\centering
\includegraphics[scale=0.3]{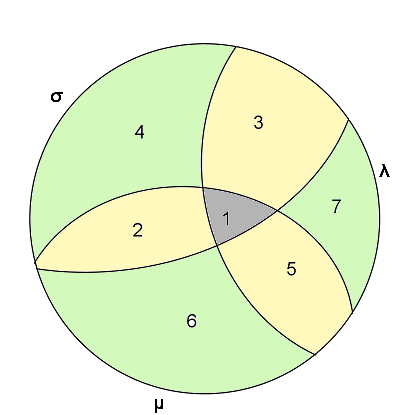}
\caption{\small{Representative diagram of $\sigma$, $\lambda$ and $\mu$}}
\label{diagram}
\end{figure}

Now we are going to prove the equality (\ref{eq9}). We have $[\sigma \cup \mu, \lambda]=[\sigma,\lambda][\mu,\lambda]$, if \linebreak $\sigma \cap \mu =\emptyset$. By the Diagram (Fig.\ref{diagram}), we can rewrite $\sigma, \lambda$ and $\mu$ as follows: $\sigma = (1234)$, $\lambda =(1357)$ and $\mu=(1256)$. So the right and left parts of (\ref{eq9}) can be rewritten as a product of $[i,j]$, $i,j=1,\dots,7$. As $[i,j]^{2}=1$, the equality (\ref{eq9}) is valid.

To prove the firt relation in (\ref{eq10}), we note that $\left[\lambda \cup \tau, \mu, \sigma\right]=\left[\lambda,\mu, \sigma\right]\left[\tau, \mu, \sigma\right],$ if $\lambda \cap \tau =\emptyset .$ Therefore,
 \begin{eqnarray*}[\lambda \Delta \tau, \mu, \sigma]=&
 [\lambda \setminus \tau, \mu, \sigma][\lambda \cap \tau, \mu, \sigma][\lambda \cap \tau, \mu, \sigma][\tau \setminus \lambda, \mu, \sigma] = [\lambda, \mu, \sigma][\tau, \mu, \sigma].
 \end{eqnarray*}
 
Analogously, we prove $\left[\sigma, \lambda \Delta \tau, \mu\right]=[\sigma, \lambda, \mu][\sigma, \tau, \mu]$ and $\left[\sigma, \mu, \lambda \Delta \tau\right]=[\sigma, \mu, \lambda][\sigma, \mu, \tau] \nonumber$.

We are going to prove (\ref{eq11}) and (\ref{eq12}) in the cases $\sigma \cap \mu = \sigma \cap \lambda = \mu \cap \lambda = \emptyset$. First we consider $i,k \in \sigma$ and $j \in \lambda$. Suppose $i < k$. We have to analyze  three cases: $1)\; i < j < k$, $2)\; i < k < j$ and $3)\; j < i < k$. In the first case, the product $i\wedge j\wedge k$ appears twice in $[\sigma,\lambda,\sigma][\sigma,\sigma,\lambda]$, as a factor of $[\sigma,\lambda,\sigma]$  and as a factor of $[\sigma,\sigma,\lambda]$. In the second case we have a factor of the form $i\wedge k \wedge j . k\wedge i \wedge j$ in $[\sigma,\sigma,\lambda]$ and there are not nontrivial factors in $[\sigma,\lambda,\sigma]$ with $i,j,k$, while in the third case we have the factor $k\wedge j \wedge i . i\wedge j \wedge k$ in $[\sigma,\lambda,\sigma]$ and there are not nontrivial  factors in $[\sigma,\sigma,\lambda]$, with $i,j,k$. Those cases where $k < i$ are analogous. Therefore, the relation (\ref{eq11}) is valid in the particular case $\sigma \cap \lambda = \emptyset$.

Now let $i,j,k$ be elements of $\sigma, \mu$ and $\lambda$ respectively. We consider the same three cases to prove the relation (\ref{eq12}) in the particular case. In any case we will have the factor $i \wedge j \wedge k$ exactly twice in $\left[\sigma, \lambda, \mu\right][\mu, \lambda,\sigma][\sigma,\mu, \lambda][\mu, \sigma, \lambda]$. For example, in the first case there are not factors with $i,j,k$ in $[\sigma,\lambda,\mu]$ and $[\mu,\lambda,\sigma]$, and there is one factor $i\wedge j\wedge k$ in $[\sigma, \mu, \lambda]$ and one factor $j\wedge i \wedge k$ in $[\mu, \sigma, \lambda]$.

To prove the relation (\ref{eq11}) in general case, we consider again $\sigma=(1234)$ and $\lambda=(1357)$. The term $[\sigma,\lambda,\sigma]=[1324,1357,1324]$ can be rewritten as $$[13,13,13][13,13,24][13,57,13][13,57,24][24,13,13][24,13,24][24,57,13][24,57,24].$$
In the same way the term $[\sigma,\sigma,\lambda]=[1324,1324,1357]$ can be rewritten as 
$$[13,13,13][13,13,57][13,24,13][13,24,57][24,13,13][24,13,57][24,24,13][24,24,57].$$
By definition, $13 \cap 57 = 13 \cap 24 = 24 \cap 57 = \emptyset$. So we obtain the relation (\ref{eq11}) in the general case.

Next, by the relation (\ref{eq11}) we obtain that $[\sigma \Delta \mu, \lambda, \sigma \Delta \mu][\sigma \Delta \mu, \sigma \Delta \mu, \lambda]=1$ and by the relation (\ref{eq10}) we prove the relation (\ref{eq12}) in the general case. Finally, the relation (\ref{eq13}) is obtained directly from relations (\ref{eq10}), (\ref{eq11}) and (\ref{eq12}).

\end{proof}

\begin{theorem}
The set $\cal F$ defined above is a Moufang loop in the variety $\cal E$.
\end{theorem}
\begin{proof}
By the definition of ${\cal F}$ we have ${\cal F} \in {\cal E}$. We consider $\sigma, \mu, \lambda \in V$. We are going to prove the following Moufang identity:
\begin{equation}
(\sigma. \mu).(\lambda. \sigma)= (\sigma.(\mu . \lambda)).\sigma  \label{eq15}
\end{equation}

By the Definition (\ref{eq3}) we have: $$(\sigma. \mu).(\lambda. \sigma)= ((\sigma \Delta \mu).\{\sigma,\mu\}.\left[\sigma,\mu \right].\left[\sigma,\mu, \mu \right]).((\lambda \Delta \sigma).\{\lambda,\sigma\}.\left[\lambda,\sigma \right].\left[\lambda,\sigma, \sigma \right])$$

And by (\ref{eq3}) and (\ref{eq4}) we have: 
\begin{eqnarray*}(\sigma. \mu).(\lambda. \sigma)&=& (\mu \Delta \lambda ).\{\sigma \Delta \mu,\lambda \Delta \sigma\}.\{\sigma,\mu\}.\{\lambda,\sigma\}.\\&& \left[\sigma \Delta \mu,\lambda \Delta \sigma\right].\left[\sigma,\mu \right].\left[\lambda,\sigma \right].\left[\sigma,\mu, \mu \right].\\&&\left[\sigma \Delta \mu,\lambda \Delta \sigma,\lambda \Delta \sigma\right].\left[\lambda,\sigma, \sigma \right]
\end{eqnarray*}

On the other hand,
\begin{eqnarray*}(\sigma.(\mu . \lambda)).\sigma &=&(\sigma . (\mu \Delta \lambda).\{\mu, \lambda\}.\left[\mu, \lambda\right].\left[\mu, \lambda, \lambda\right]).\sigma\\&=&(\sigma \Delta (\mu \Delta \lambda).\{\sigma, \mu \Delta \lambda\}.\{\mu, \lambda\}.\left[\mu, \lambda\right].\\&&\left[\sigma, \mu \Delta \lambda\right].\left[\sigma, \mu \Delta \lambda, \mu \Delta \lambda\right].\left[\mu, \lambda, \lambda\right]).\sigma\\&=&(\mu \Delta \lambda).\{\sigma \Delta \mu \Delta \lambda, \sigma\}.\{\sigma, \mu \Delta \lambda\}.\{\mu, \lambda\}.\\&&
\left[\sigma \Delta \mu \Delta \lambda, \sigma\right].\left[\sigma, \mu \Delta \lambda\right].\left[\mu, \lambda\right].\left[\mu,\lambda,\lambda\right].\\&&\left[\sigma \Delta \mu \Delta \lambda, \sigma, \sigma\right].\left[\sigma, \mu \Delta \lambda, \mu \Delta \lambda\right]
\end{eqnarray*}

Now the equality (\ref{eq15}) follows from the Lemma \ref{lemma2.3}.

\end{proof}

We consider a subloop of $\cal F$, denoted by ${\cal F}_{n}$ and with generators $\left\{i \in I_{n}\right\}$, where $I_{n}=\left\{1,\dots,n\right\}$. 

The next proposition contains important results that we are going to use throughout this paper. The demonstration of the first and second result comes directly from (\cite{4}, Theorem 2). The last result is proved easily.

\begin{proposition} \label{prop2.5}Let $F$ be a Moufang loop.
\begin{enumerate}
\item If $(x,y,z)^{2}=1$ and all the commutators and associators of $F$ are central, then
      \begin{eqnarray*}
      [xy,z]=[x,z][y,z](x,y,z).
      \end{eqnarray*}
\item If the commutators and associators of $F$ are  central, then
      \begin{eqnarray*}
      (wx,y,z)=(w,y,z)(x,y,z).
       \end{eqnarray*}
       
       \item  If the squares and commutators of a Moufang loop $F$ are central, then 
       \begin{eqnarray*}(xy)^{2}=x^{2}y^{2}[x,y].\end{eqnarray*}
\end{enumerate}
\end{proposition}

\begin{lemma}\label{lemma2.6}
Let $F_{n}$ be a free loop in ${\cal E}$, with free generator set $\{x_{1},\dots,x_{n}\}$.

Then for all $z \in {\cal Z}(F_{n})$, there are $\xi_{1},\dots,\xi_{n}$, $\xi_{ij}$, $\xi_{ijk}\in\{0,1\}$, with $i,j,k=1,\dots,n$ such that
\begin{equation}
z = \prod_{i=1}^{n}(x_{i}^{2})^{\xi_{i}}.\prod_{i<j}[x_{i},x_{j}]^{\xi_{ij}}.\prod_{i<j<k}(x_{i},x_{j},x_{k})^{\xi_{ijk}}. \label{eq16}
\end{equation}
\end{lemma}

\begin{proof} Let $\cal{N}$ be a central subloop of $F_n$ generated by the set $\{x_{i}^{2},[x_{i},x_{j}],(x_{i},x_{j},x_{k})|i,j,k\in I_n\}.$

By the Proposition \ref{prop2.5} we have:
\begin{eqnarray*}
(x_{i}x_{j})^{2}&=&x_{i}^{2}x_{j}^{2}\left[x_{i},x_{j}\right];\\
\left[x_{i}x_{j},x_{k}\right]&=&\left[x_{i},x_{k}\right]\left[x_{j},x_{k}\right](x_{i},x_{j},x_{k});\\
(x_{i}x_{j},x_{k},x_{p})&=&(x_{i},x_{k},x_{p})(x_{j},x_{k},x_{p}).
\end{eqnarray*}
Hence $F_n/\cal{N}$ is an abelian group of expoent two. Then ${\cal{N}}={\cal{Z}}(F_{n})$ and the Lemma is proved.
\end{proof}

\begin{theorem}\label{theorem2.7}${\cal F}_{n}$ is a free loop in the variety ${\cal E}$ of rank $n$.
\end{theorem}

\begin{proof}
Let $F_{n} \in {\cal E}$, $F_{n}=\left\langle x_{1},\dots,x_{n}\right\rangle$ be a free loop with $n$ generators, then by the definition of free loop  there is a unique surjective morphism $\varphi: F_{n} \longrightarrow {\cal F}_{n}$ such that $\varphi(x_{i})=i.$

To show that ${\cal F}_{n}$ is free, we just need to prove that $\varphi$ is injective. It is clear that $ker\varphi\subseteq {\cal{Z}}(F_{n}).$ By the Lemma \ref{lemma2.6} any $z\in ker\varphi\subseteq {\cal{Z}}(F_{n})$ has the form  given in (\ref{eq16}). Then $\varphi(z)\not=1$ by definition of $\cal F.$ Therefore, ${\cal F}_{n}$ is free loop.

\end{proof}

\begin{corollary}
(1) For any code loop $L$ of rank $n$ there is a homomorphism \linebreak $\varphi:{\cal F}_{n}\longrightarrow L$ such that $\varphi({\cal F}_{n})=L$ and codim$_{{\cal{Z}}({\cal F}_{n})}ker(\varphi)=1.$

(2) For all $\textbf{F}_{2}$-subspace $T \subset \cal{Z}({\cal F}_{n})$ of codimension $1$ we have a code loop $L(T)={\cal F}_{n}/T.$

(3) The loop $L(T)$ is a group if and only if $T \supseteq U_{n} = ({\cal F}_{n},{\cal F}_{n},{\cal F}_{n})={\cal F}_{n}\cap U$, 
where $U=(\cal F, \cal F, \cal F).$
\end{corollary}

\begin{proposition}\label{prop2.9} Let $T_{1}$ and $T_{2}$ be $\textbf{F}_{2}$-subspaces of ${\cal{Z}}({\cal F}_{n})$ of codimension $1$, so
$L(T_{1})\cong L(T_{2})$ if and only if there is an automorphism $\varphi$ of ${\cal F}_{n}$ such that $T_{1}^{\varphi}=T_{2}.$
\end{proposition}

\begin{proof}
Let $\varphi:{\cal F}_{n}\longrightarrow {\cal F}_{n}$ be an automorphism such that $T_{1}^{\varphi}=T_{2},$ where $T_{1}$ and $T_{2}$ are $\textbf{F}_{2}$-subspaces of $ {\cal{Z}}({\cal F}_{n})$ of codimension $1$. We define $\bar{\varphi}:{\cal F}_{n}/T_{1} \longrightarrow {\cal F}_{n}/T_{2}$ by $\bar{\varphi}(\bar{x})=\overline{\varphi(x)}$, for all $\bar{x} \in {\cal F}_{n}/T_{1}$. Clearly, $\varphi$ is an isomorphism between code loops.

Now we construct the following commutative diagram: 
\begin{center}
\begin{math}
\xymatrix{
1 \ar[r]& T_{1}\ar[r]^{i_{1}}\ar[d]^{\varphi=\tilde{\sigma}|_{T_{1}}} &  {\cal F}_{n}\ar[r]^{\pi_{1}}\ar[d]^{\tilde{\sigma}} & L_{1} \ar[r]\ar[d]^{\sigma}& 1\\
1 \ar[r]& T_{2}\ar[r]^{i_{2}}&  {\cal F}_{n}\ar[r]^{\pi_{2}} & L_{2}\ar[r] & 1 }\end{math}
\end{center}

We suppose that $L_{j}={\cal F}_{n}/T_{j}$, $\sigma: L_{1}\longrightarrow L_{2}$ is an isomorphism, $\pi_{j}$ is a surjective homomorphism and $i_{j}$ is the inclusion application, for $j=1,2$. Thus, 
$$ \xymatrix{1 \ar[r]& T_{j}\ar[r]^{i_{j}} &  {\cal F}_{n}\ar[r]^{\pi_{j}} & L_{j} \ar[r]& 1}$$ is a short exact sequence, for $j=1,2$.

Let $x_1,\dots,x_n$ be a set of free generators of ${\cal F}_{n}$ and $v_i=\sigma(\pi_1(x_i))\in L_2.$ Since $\pi_2$ is surjective, there exists $y_i\in {\cal F}_{n}$ such that $\pi_2(y_i)=v_i.$ Then there exists a homomorphism $\tilde{\sigma}:{\cal F}_{n}\to {\cal F}_{n}$ such that
$\tilde{\sigma}(x_i)=y_i.$

So, since both the right part of the diagram is commutative by construction, we have $\varphi(T_1)=\tilde{\sigma}|_{T_1}(T_1)\subset T_2.$

We have that both $T_{1}$ and $T_{2}$ has the same dimension, then we just need to prove that $\varphi$ is surjective. In fact, let $x\in T_{2}$, then $i_{2}(x)=x \in {\cal F}_{n}$, but $\tilde{\sigma}$ is surjective, thus, there exists $i \in {\cal F}_{n}$ such that $\tilde{\sigma}(i)=x$. On the other hand, $x \in Ker(\pi_{2})=Im(i_{2})$, because the sequences of lines of the previous diagram are short exact sequences. Hence $\pi_{2}(x)=T_{2}$, that is, $\pi_{2}(\tilde{\sigma}(i))=T_{2}$. As the right part of the diagram is commutative, we have $\sigma(\pi_{1}(i))=T_{2}$. Then $\pi_{1}(i) \in Ker(\sigma)=\{T_{1}\}$, because $\sigma$ is bijective, that is, $\pi_{1}(i) = T_{1}$. Hence, $i \in Ker(\pi_{1})=Im(i_{1})=T_{1}$, as we wanted.

\end{proof}

We denote by ${\cal L}_{n}$ the set of all subspaces $T \subset {\cal{Z}}({\cal F}_{n})$ of codimension $1$ such that $T \not \supseteq U_{n}$. The group $G_{n}=\mbox{\normalfont Aut}{{\cal F}_{n}}$, group of automorphisms of ${{\cal F}_{n}}$, acts on ${\cal L}_{n}$.

Let $L_{n}$ be the set of the correspondents $G_{n}$-orbits. We denote by $O_{T}$ the orbit of $T$, that is,  $O_{T}=\left\{T^{\sigma}| \sigma \in G_{n}\right\}$.

\begin{corollary}
The correspondence $T \longrightarrow L(T)$ gives a bijection between $L_{n}$ and the set of the isomorphism classes of nonassociative code loops of rank $n$.
\end{corollary}

Let $L$ be a code loop with generator set $X=\left\{x_{1},\dots,x_{n}\right\}$ and center $\left\{1,-1\right\}$. Then we define the characteristic vector of $L$, denoted by $\lambda_{X}(L)$ or $\lambda(L)$, by \[\lambda(L)=(\lambda_{1},\dots,\lambda_{n};\lambda_{12},\dots,\lambda_{1n},\dots,\lambda_{(n-1)n};\lambda_{123},\dots,\lambda_{12n},\dots,\lambda_{(n-2)(n-1)n}),\] where $\lambda_{i}, \lambda_{ij},\lambda_{ijk} \in \textbf{F}_{2}$, $(-1)^{\lambda_{i}}=x_{i}^{2},\;\; (-1)^{\lambda_{ij}}=[x_{i},x_{j}]\; \;\mbox{e}\; \; (-1)^{\lambda_{ijk}}=(x_{i},x_{j},x_{k})$.
 
  With $\lambda(L)\neq 0$ we can associate a subspace $T=T_{\lambda}$ in ${\cal{Z}}({\cal F}_{n})$ of codimension 1 as follows:
 
 We choose an element $\lambda_{ijk}=1$ and we denote it by $x=x_{i}\wedge x_{j}\wedge x_{k}$. \linebreak Then 
$T_{\lambda}=\textbf{F}_{2}\left\{x_{i}^{2},xx_{j}^{2}| \lambda_{i}=0, \lambda_{j}=1\right\} + \textbf{F}_{2}\left\{[x_{p},x_{q}],x[x_{i},x_{j}] | \lambda_{pq}=0, \lambda_{ij}=1\right\}$ +\\ $+ \textbf{F}_{2}\left\{(x_{i},x_{j},x_{k}), x(x_{p},x_{q},x_{l}) | \lambda_{ijk}=0, \lambda_{pql}=1\right\}$, by definition.

The following proposition is a simple corollary of definitions.
\begin{proposition}
 The application $\lambda \longrightarrow T_{\lambda}$ is a bijection between the set of  characteristic vectors of nonassociatives code loops of rank $n$ and the subspaces $T$ in ${\cal Z}={\cal Z}({\cal F}_{n})$ of codimension 1  such that $U_{n} \not\subset T$.  
\end{proposition}

 So we have a $G_{n}$-action over the set of the characteristic vectors. The action is defined by ${\lambda}^{\varphi}=\mu$ if and only if ${T_{\lambda}}^{\varphi}=T_{\mu}$, for every $\varphi \in G_{n}=\mbox{\normalfont Aut}({\cal F}_{n})$. 
 
 Without loss of generality we can assume $\lambda_{123}=1$. In this case, we have $2^{m}$ characteristic vectors where $m=n+\frac{n(n-1)}{2}+\frac{n(n-1)(n-2)}{6}-1$. Let $\left\{O_{1},\dots,O_{k}\right\}$ be a set of $G_{n}-$orbits of characteristic vectors of the code loops of rank $n$. Then
 \begin{equation}
 \sum_{i=1}^{k}|O_{i}|=2^{m}.
 \end{equation}
 
 For a given code loop $L$, we denote by $\mbox{\normalfont Aut}L$ the group of automorphisms of $L$ and we define $\mbox{\normalfont Out} L = \mbox{\normalfont Aut} L/N(\mbox{\normalfont Aut} L)$, the group of outer automorphisms of $L$. By definition, $N(\mbox{\normalfont Aut} L)=\left\{\phi \in \mbox{\normalfont Aut} L | \phi(x)=\pm x, \forall x \in L\right\}.$  

  Let $G_{n}^{i}=\left\{\varphi \in G_{n} | \lambda^{\varphi} = \lambda\right\}.$ 
  We note that $\varphi \in G_{n}^{i}$ induces an automorphism of the correspondent code loop $L_{i}={\cal F}_{n}/T_{\lambda}$. 
  
 For the next proposition, we define  ${\cal Z}^{n}=\{\varphi\in G_{n}|\forall i, \,\varphi(i)=iz_{i},\, z_i\in {\cal{Z}}({\cal F}_n)\}\simeq {\cal{Z}}({\cal F}_n)^n$ and we denote by $GL_{n}(2)$ the General Linear Group of degree $n$ over a finite field with 2 elements. 
 
  \begin{proposition} Let $\overline{GL_{n}(2)}=\{\varphi \in G_{n}|i^{\varphi}\subseteq V, \; \mbox{for all}\; \; i \in \{1,\dots,n\}\}\subset G_{n}.$
   Then:
  \begin{enumerate}
  \item $G_{n}=\overline{GL_{n}(2)}.{\cal Z}^{n}$;
  \item $G_{n}/{\cal Z}^{n}\simeq GL_{n}(2)=\mbox{\normalfont Aut}_{{\bf F}_{2}}\;V$ and 
  \item $G_{n}^{i}/{\cal Z}^{n}\simeq \mbox{\normalfont Out}(L_{i}) $.
  \end{enumerate}
  
	\end{proposition}
  \begin{proof}
  Let $\varphi \in G_{n}$, so $\varphi(i)=\sigma_{i}.z_{i}$, for $\sigma_{i}\in V=\{1,\dots,n\}$ and $z_{i}\in {\cal Z}$. Thus, for all $\varphi \in G_{n}$, there are $\{\sigma_{1},\dots,\sigma_{n}\}$ a corresponding basis of $V$ and $(z_{1},\dots,z_{n})\in {\cal Z}^{n}$. We have that ${\cal F}_{n}/{\cal Z}({\cal F}_{n})\simeq V=\{1,\dots,n\}$ is a vector ${\bf F}_{2}$-space and, for $\varphi \in G_{n}$, we have an automorphism $\tilde{\varphi}:{\cal F}_{n}/{\cal Z}({\cal F}_{n})\longrightarrow {\cal F}_{n}/{\cal Z}({\cal F}_{n})$ such that $i^{\tilde{\varphi}}=\sigma_{i}$. Then $\tilde{\varphi} \in GL(V)=GL_{n}(2)$.  
  
  We are going to show now that $G_{n}^{i}/{\cal Z}^{n}\simeq \mbox{\normalfont Out}(L_{i})$. In fact, let $\varphi \in G_{n}^{i}$, then $T_{\lambda}^{\varphi}=T_{\lambda}$, hence $\tilde{\varphi}:{\cal F}_{n}/T_{\lambda}\longrightarrow {\cal F}_{n}/T_{\lambda}$ is an isomorphism of code loops. The homomorphism $\pi_{i}:G_{n}^{i}\longrightarrow \mbox{\normalfont Aut}(L_{i})$ such that $\pi_{i}(\varphi)=\tilde{\varphi}$, is surjective (by the Proposition \ref{prop2.9}) with kernel $K_{i}=\{\tau \in G_{n}^{i}|\tau(x_{j})=x_{j}v_{j}, v_{j}\in T_{\lambda}\}$.
	
	Consider also the surjective homomorphism $\psi_{i}: \mbox{\normalfont Aut}(L_{i})\longrightarrow \mbox{\normalfont Out}(L_{i})$. Then  $\pi=\psi_{i}\circ \pi_{i}: G_{n}^{i}\longrightarrow \mbox{\normalfont Out}(L_{i})$ is a surjective homomorphism. We have ${\cal Z}^{n}\subseteq ker(\pi)$.
Now, let $\varphi \in ker(\pi)$. If $\varphi \notin {\cal Z}^{n}$, there  is $i^{\varphi}\neq iz_{i}$, for all $ z_{i}\in {\cal Z}$, then for $\pi(\varphi)\in \mbox{\normalfont Out}(L_{i})$, we have \linebreak $\pi(\varphi):i\longmapsto \pm i^{\varphi} \neq \pm i$ $(i\; mod(T_{\lambda}))$, hence $\pi \varphi \neq 1$ in $\mbox{\normalfont Out}(L_{i})$. Thus, $ker(\pi)={\cal Z}^{n}$ and, therefore, $G_{n}^{i}/{\cal Z}^{n}\simeq \mbox{\normalfont Out}(L_{i})$. 
  \end{proof}
  
	We denote $G_{n}^{i}/{\cal Z}^{n}$ by $\mbox{\normalfont Stab}(\lambda)$.
  
  \begin{definition}
A $3-$space is a ${\bf F}_{2}-$space $V$ with an antisymmetric trilinear form $(\;,\; ,\;)$ with values in ${\bf F}_{2}$.

A $2-$space is a $3-$space $V$ with a bilinear map $[\;,\;]:V \times V \longrightarrow {\bf F}_{2}$ such that $[v,wu]=[wu,v]=[v,w][v,u](v,w,u).$
\end{definition}

For a given code loop $L$, we can associate a $2$-space $V=L/{\cal Z}(L)$ with a $3$-form induced by the associator and the correspondent antisymmetric aplication induced by the commutator.

Let $V$ be a $3-$space. By definition, the nucleus of $V$ is the \linebreak subspace $N_{3}(V)=\left\{v \in V | (v,V,V)=1\right\}.$ Similarly, for a $2-$space $V$, we define \linebreak $N_{2}(V)=\left\{v \in V | (v,V,V)=[v,V]=1\right\}.$

\begin{definition}
Let $L_{1}$ and $L_{2}$ be code loops. We define the product of $L_{1}$ and $L_{2}$ as a code loop $L_{1}\ast L_{2}=\dfrac{L_{1}\times L_{2}}{D}$, where $D=\{1\times 1, (-1)\times(-1)\}$ is a diagonal of the center of $L_{1}\times L_{2}$. 

Let $L$ be an other code loop, we will write  $L_{1}\sim_{L}L_{2}$ if $L\ast L_{1} \simeq L\ast L_{2}$ and we will say that $L_{1}$ and $L_{2}$ are $L$-equivalents.

\end{definition}

For the following, $\mathbb{Z}_{m}$ is the (additive) group of integers modulo m.

\begin{lemma}\label{lemma2.15}
Two code loops $L_{1}$ and $L_{2}$ are isomorphic as $2$-spaces if and only if $L_{1}\sim_{\mathbb{Z}_{4}}~L_{2}$.

\begin{proof}
Let $L_{1}$ and $L_{2}$  be $\mathbb{Z}_{4}$-equivalents. Therefore, there exists an isomorphism \linebreak $\sigma: \mathbb{Z}_{4}\ast L_{1}\longrightarrow \mathbb{Z}_{4}\ast L_{2}$. We define a linear map $\tau:L_{1}\longrightarrow L_{2}$ by $\tau(v)=w \in L_{2}$ if $\sigma(v)=aw$, where $a \in \mathbb{Z}_{4}$. Since by definition $[a,L_{2}]=(a,L_{2},L_{2})=1$, then $\tau$ is an isomorphism of $L_{1}$ and $L_{2}$ as $2$-spaces.

Conversely, if $\tau:L_{1}\longrightarrow L_{2}$ is an isomorphism of code loops as $2$-spaces, so the application $\sigma: \mathbb{Z}_{4}\ast L_{1} \longrightarrow \mathbb{Z}_{4}\ast L_{2}$, where $\sigma(a\times v)=b\times \tau(v)$ and $b=a$ if $v^{2}=\tau(v)^{2}$, $b=ac$ if $v^{2}\neq \tau(v)^{2}$, $a,b,c \in \mathbb{Z}_{4}$, $c^{2}=-1$,  is an isomorphism.
\end{proof}

\end{lemma}

\section{Automorphisms of Code Loops}

Classification of all code loops untill rank 4 using the definition of characteristic vector has been done in \cite{1}.
In this section we will illustrate the notions introduced in this case.

\subsection{Code loops of rank 3}

Let $L$ be a nonassociative code loop of rank $3$ with generators $a,b,c.$ Then\linebreak $(a,b,c)=-1$. We associate to $L$ the characteristic  vector  $\lambda(L)$, defined by\linebreak $\lambda(L) = (\lambda_1,...,\lambda_6)$ where \; $a^{2}=(-1)^{\lambda_1}$,$b^{2}=(-1)^{\lambda_2}$, $c^{2}=(-1)^{\lambda_3}$,$\left[a,b\right]=(-1)^{\lambda_4}$, $\left[a,c\right]=(-1)^{\lambda_5}$, $\left[b,c\right]=(-1)^{\lambda_6}.$

\begin{theorem}\label{theorem3.1}
Consider $C_{1}^{3},...,C_{5}^{3}$ the code loops with the following characteristic vectors:
\begin{eqnarray*}
\begin{array}{lllll}
\lambda(C_{1}^{3})=(1,1,1,1,1,1), &&\lambda(C_{2}^{3})=(0,0,0,0,0,0),&&\lambda(C_{3}^{3})=(0,0,0,1,1,1),\\
\lambda(C_{4}^{3})=(1,1,0,0,0,0),&&\lambda(C_{5}^{3})=(1,0,0,0,0,0).&&\\
\end{array}
\end{eqnarray*}
Then any two loops from the list $\left\{C_{1}^{3},...,C_{5}^{3}\right\}$  are not isomorphic and all nonassociative
code loop of rank $3$ is isomorphic to one of this list.
\end{theorem}

\begin{proof}
Let $a,b,c$ be the generators of the nonassociative code loop $L$ of rank $3$. Thus $(a,b,c)=-1$. There are two possibilities:
\begin{enumerate}
	\item $\left[a,b\right] = 1;$
	\item $\left[a,b\right] = \left[a,c\right] = \left[b,c\right] = -1.$
\end{enumerate}

In the first case, we can choose generators such that $\left[a,b\right] = \left[a,c\right] = \left[b,c\right] = 1$. In fact, if $\left[a,c\right] = -1$,  we can take $x=bc$ and we will have $\left[a,x\right] = \left[a,b\right]\left[a,c\right](a,b,c) = 1,$ (see Theorem 2, \cite{4}). 

If $\left[a,b\right]=\left[a,c\right]=1$ and since $\left[b,c\right]=-1$, we obtain $\left[a,ab\right]=\left[a,c\right]=\left[ab,c\right]=1.$

Fixed the generators $a,b,c$ such that $\left[a,b\right] = \left[a,c\right] = \left[b,c\right] = 1$, we will have for $\lambda(L)$ four possibilities: $(0,0,0,0,0,0),(1,1,0,0,0,0),(1,0,0,0,0,0),(1,1,1,0,0,0).$

Notice that the loops that correspond to the last two characteristic vectors are isomorphics. In fact, let $L_1$ be the code loop with generator set $ X = \left\{a,b,c\right\}$ and with corresponding characteristic vector $\lambda_{X}(L_1) = (1,1,1,0,0,0)$. Then, for the generator set $Y = \left\{a,ab,ac\right\}$, we have $a^{2}=-1, (ab)^{2}=1, (ac)^{2}=1$, that is, $\lambda_{Y}(L_1) = (1,0,0,0,0,0)$.

In the second case, there are two possibilities:
\begin{description}
	\item[i)]  $a^{2}=b^{2}=c^{2}=-1$: Here we have $\lambda(L)=(1,1,1,1,1,1)$.
	\item[ii)] $a^{2}=1:$ If we assume $b^{2}=-1$ or $c^{2}=-1$, we will have  $(ac)^{2}=(ab)^{2}=1$. Hence we can choose generators $a,b,c$ such that $a^{2}=b^{2}=c^{2}=1$.\\
	Therefore, $\lambda(L)=(0,0,0,1,1,1).$
\end{description}

\end{proof}

\begin{proposition}\label{prop3.2}

Let $O^{3}_{1},...,O^{3}_{5}$ be the orbits of the characteristic vectors associated to code loops $C_{1}^{3},...,C_{5}^{3}$, respectively. Then

$\begin{array}{l} O^{3}_1 = \{(111111)\};\end{array} $

$\begin{array}{l} O^{3}_2 = \{(000000),(000001),(000010),(000100),(001011),(010101),(100110)\};\end{array} $

$\begin{array}{l} O^{3}_3 = \{(000111),(001111),(010111),(011111),(100111),(101111),(110111)\};\end{array} $

 $\begin{array}{l} \begin{aligned}
  O^{3}_4 = \{&(001100),(010010),(011000),(011001),(011011),(011101),(100001),\\
&(101000),(101010),(101011),(101110),(110000),(110100),(110101),\\
&(110110),(111001),(111010),(111011),(111100),(111101),(111110) \};\\
 \end{aligned} \end{array} $

 $ \begin{array}{l} \begin{aligned}
 O^{3}_5 =\{&(000011),(000101),(000110),(001000),(001001),(001010),(001101),\\ &(001110),(010000),(010001),(010011),(010100),(010110),(011010),\\
&(011100),(011110),(100000),(100010),(100011),(100100),(100101),\\
&(101001),(101100),(101101),(110001),(110010),(110011),(111000) \}.\\ 
\end{aligned} \end{array}  $

We note that $|O_1^{3}| = 1, |O_2^{3}|=|O_3^{3}|=7, |O_4^{3}|=21, |O_5^{3}|=28$ and $\displaystyle \sum_{i=1}^{5}|O_i^{3}|=64=2^{6}.$
Therefore, $|\mbox{\normalfont Out} C_{1}^{3}|=|GL_{3}(2)|=168 ,$
          $|\mbox{\normalfont Out} C_{2}^{3}|=|\mbox{\normalfont Out} C_{3}^{3}|=24,$
          $|\mbox{\normalfont Out} C_{4}^{3}|=8$ and $|\mbox{\normalfont Out} C_{3}^{3}|=6.$
\end{proposition}

\begin{proof}
Let $(a,b,c)$ be a fixed set of generators for each code loop $C_{i}^{3}$, $i=1,\dots,5$. We consider the characteristic vector $\lambda(C_{i}^{3})$ as representative of each orbit $O_{i}^{3}$, $i=i,\dots,5$. To find all characteristic vectors belonging to the orbit of each representative, we have to find the different set of generators formed from the set of generators $(a,b,c)$ fixed. In the Table \ref{tabqccaso3}, we present all the values of squares and commutators obtained from each characteristic vector, that is, the squares of each generator element and the commutator between each pair of generators elements. Each column $i$ corresponds to a code loop $C_{i}^{3}$, $i=1,\dots,5$.
{\scriptsize \begin{table*}[!h]
	\centering
			\caption{Squares and Commutators from fixed generators $a,b,c$ for each $C_{i}^{3}$ }
		\label{tabqccaso3}
		\begin{tabular}{|r||r|r|r|r|r|}
			\hline
        & $1$  & $2$ & $3$  & $4$ & $5$   \\ \hline \hline
$a^2$   & $-1$ & $1$ & $1$  & $-1$ & $-1$    \\ \hline
$b^2$   & $-1$ & $1$ & $1$  & $-1$ & $1$     \\ \hline
$c^2$   & $-1$ & $1$ & $1$  & $1$ & $1$     \\ \hline
$[a,b]$ & $-1$ & $1$ & $-1$ & $1$ & $1$    \\ \hline
$[a,c]$ & $-1$ & $1$ & $-1$ & $1$ & $1$   \\ \hline
$[b,c]$ & $-1$ & $1$ & $-1$ & $1$ & $1$   \\ \hline

		\end{tabular}
\end{table*}}

 Let $\varphi \in GL_{3}(2)$ an automorphism of $ C_{i}^{3}$. Then we find a new set of generators of $C_{i}^{3}$, denoted by $(x,y,z)$, that is, $(a^{\varphi}=x, b^{\varphi}=y, c^{\varphi}=z)$. We see in the Table  \ref{tabqccaso3-2} all the possibilities for $x,y$ and $z$, and the squares of $x, y, z$ and the commutators between   each pair of elements. 
{\scriptsize \begin{table}[!h]
\centering
\caption{Squares and Commutators of $x,y$ and $z$}
\label{tabqccaso3-2}
\begin{tabular}{cc}
\begin{tabular}{|l|r|r|r|r|r|}
			\hline
          & $1$  & $2$ & $3$ & $4$ & $5$   \\ \hline 
$(ab)^2$  & $-1$ & $1$ & $-1$ & $1$ & $-1$  \\ 
$(ac)^2$  & $-1$ & $1$ & $-1$ & $-1$ & $-1$   \\ 
$(bc)^2$  & $-1$ & $1$ & $-1$ & $-1$ & $1$  \\ 
$(abc)^2$ & $-1$ & $-1$ & $1$ & $-1$ & $1$    \\ 
$[a,ab]$  & $-1$ & $1$ & $-1$ & $1$ & $1$    \\ 
$[a,ac]$  & $-1$ & $1$ & $-1$ & $1$ & $1$    \\ 
$[a,bc]$  & $-1$ & $-1$ & $-1$ & $-1$ & $-1$    \\ 
$[a,abc]$ & $-1$ & $-1$ & $-1$ & $-1$ & $-1$    \\ 
$[b,ab]$  & $-1$ & $1$ & $-1$ & $1$ & $1$    \\ 
$[b,ac]$  & $-1$ & $-1$ & $-1$ & $-1$ & $-1$    \\ 
$[b,bc]$  & $-1$ & $1$ & $-1$ & $1$ & $1$    \\ \hline
		\end{tabular} & \begin{tabular}{|l|r|r|r|r|r|}
			\hline
          & $1$  & $2$ & $3$ & $4$ & $5$   \\ \hline 
$[b,abc]$ & $-1$ & $-1$ & $-1$ & $-1$ & $-1$    \\ 
$[c,ab]$  & $-1$ & $-1$ & $-1$ & $-1$ & $-1$    \\
$[c,ac]$  & $-1$ & $1$ & $-1$ & $1$ & $1$    \\ 
$[c,bc]$  & $-1$ & $-1$ & $1$ & $1$ & $1$    \\ 
$[c,abc]$ & $-1$ & $1$ & $-1$ & $-1$ & $-1$    \\ 
$[ab,ac]$ & $-1$ & $1$ & $-1$ & $1$ & $1$    \\ 
$[ab,bc]$ & $-1$ & $1$ & $-1$ & $1$ & $1$    \\ 
$[ab,abc]$& $-1$ & $-1$ & $-1$ & $-1$ & $-1$    \\ 
$[ac,bc]$ & $-1$ & $1$ & $1$ & $1$ & $1$    \\
$[ac,abc]$& $-1$ & $-1$ & $-1$ & $-1$ & $-1$    \\ 
$[bc,abc]$& $-1$ & $-1$ & $-1$ & $-1$ & $-1$    \\ \hline

		\end{tabular}
\end{tabular}
\end{table} }

In the first orbit we have only the vector $\lambda(C_{1}^{3})=(111111)$, because for any $x,y,z$ as above, we always have $x^{2}=y^{2}=z^{2}=-1$ and $[x,y]=[x,z]=[y,z]=-1$.

To determine the vectors of the orbit $O_{2}^{3}$, we observe from the Table \ref{tabqccaso3-2} that the set of generators has the following possibilities: $(x,y,z)=(a,b,c)$, $(a,b,ac)$, $(a,b,bc)$, $(a,bc,b)$, $(abc,ac,bc)$, $(ac,abc,bc)$ and $(ac,bc,abc)$, which corresponds, respectively, to the characteristic vectors  $(000000)$, $(000001)$, $(000010)$, $(000100)$, $(100110)$, $(010101)$ and $(001011)$. All these vectors  form the orbit $O_{2}^{3}$. In the same way, we obtain the vectors of $O_{3}^{3}$. The possibilities for $(x,y,z)$, in this case, are $(a,b,c)$, $(ab,a,c)$, $(a,ab,c)$, $(a,c,ab)$, $(ab,ac,a)$, $(ab,a,ac)$ and $(a,ab,ac)$, with characteristic vectors respectively given by $(000111)$, $(100111)$, $(010111)$, $(001111)$, $(110111)$, $(101111)$ and $(011111)$. Analogously, we obtain all the vectors of the orbits $O_{4}^{3}$ and $O_{5}^{3}$. 

\end{proof}

For the following statements we use the notation $GR\{x_{1},\dots,x_{n}\}$ to denote the group generated by the elements $x_{1},\dots,x_{n}$. Here, $D_{n}$ denotes the dihedral group of order $n$ and $S_{n}$ is the symmetric group on $n$ letters.

\begin{proposition}\label{prop3.3}
In the notation above, we have:
\begin{description}
\item [(1)] $\mbox{\normalfont Out} C_{1}^{3} \simeq GL_{3}(2)$

\item [(2)]$\mbox{\normalfont Out} C_{2}^{3} \simeq S_4$

\item [(3)] $\mbox{\normalfont Out} C_{3}^{3} \simeq S_4$

\item [(4)] $\mbox{\normalfont Out} C_{4}^{3} \simeq D_8$

\item [(5)] $\mbox{\normalfont Out} C_{5}^{3} \simeq S_3$

\end{description}

\end{proposition}

\begin{proof}
\begin{enumerate}
\item  Let $a,b,c$ be generators of $L = C_{1}^{3}$ with corresponding characteristic vector $\lambda = \lambda(L) = (111111)$, that is $a^{2}=b^{2}=c^{2}=-1$ and $ [a,b]=[a,c]=[b,c]=-1.$ From the Tables \ref{tabqccaso3} and \ref{tabqccaso3-2} we have $(a^{\varphi})^{2}=(b^{\varphi})^{2}=(c^{\varphi})^{2}=-1$ and $ [a^{\varphi},b^{\varphi}]=[a^{\varphi},c^{\varphi}]=[b^{\varphi},c^{\varphi}]=-1,$ for all $\varphi \in GL_{3}(2)$.

\item
We want to determine all the outer automorphisms of the code loop $L = C_{2}^{3}$. We consider $\lambda=\lambda(L)=(000000)$ as a representative characteristic vector of $L$ of the orbit $O_{2}^{3}$  and $(a,b,c)$ a set of generators for $L$ such that $a^{2}=b^{2}=c^{2}=1$ and $ [a,b]=[a,c]=[b,c]=1.$\\
We consider $\varphi_{1},\dots,\varphi_{6} \in GL_{3}(2)$ such that 
\begin{eqnarray*}
\begin{array}{ccc}
\varphi_{1}=id_{3\times 3} & \varphi_{2}=\left(\begin{array}{ccc}
1 & 0 & 0\\
0 & 0 & 1\\
0 & 1 & 0
\end{array} \right) & \varphi_{3}= \left(\begin{array}{ccc}
0 & 1 & 0\\
1 & 0 & 0\\
0 & 0 & 1
\end{array} \right)\\
\varphi_{4}=\left(\begin{array}{ccc}
0 & 0 & 1\\
1 & 0 & 0\\
0 & 1 & 0
\end{array}
 \right)  
& \varphi_{5}=\left(\begin{array}{ccc}
0 & 1 & 0\\
0 & 0 & 1\\
1 & 0 & 0
\end{array}
 \right)  
 & \varphi_{6}=\left(\begin{array}{ccc}
0 & 0 & 1\\
0 & 1 & 0\\
1 & 0 & 0
\end{array}
 \right) \\
\end{array}
\end{eqnarray*}
We have $\lambda^{\varphi_{i}}= \lambda,$ for each $i = 1,\dots,6.$ Then $\varphi_{i} \in \stab(\lambda),$ for each $i=1,\dots,6.$ We observe that $GR\left\{\varphi_{1},\dots,\varphi_{6}\right\}= \textsl{S}_{3}. $ Therefore, $\textsl{S}_{3} \subseteq \stab(\lambda).$ From the Preposition \ref{prop3.2}, in $O_{2}^{3}$ there are exactly $7$ characteristic vectors, so we have to prove that there are $24$ automorphisms fixing $\lambda$. Let $\sigma \in GL_{3}(2) $ defined by $\sigma = \left(\begin{array}{ccc}
1 & 1 & 0\\
0 & 1 & 1\\
0 & 1 & 0
\end{array} \right).$
We have ${\lambda}^{\sigma} = \lambda $, hence $\sigma \in \stab(\lambda)$. Besides, $\sigma^{3}=1 $  and ${\sigma}^{2} \in \stab(\lambda)$. Indeed, for all $i=1,\dots,6$ we have $({\sigma}{\varphi}_{i})$, $({{\sigma}^{2}}{\varphi}_{i})$, $({\varphi}_{i}{\sigma})$, $({\varphi}_{i}{{\sigma}^{2}})$ $\in \stab(\lambda)$. From these elements we obtain more 16 different stabilizers (automorphisms that fix $\lambda$). For complete our group we consider $({\sigma}{{\varphi}_{3}}){{\sigma}^{2}}$ and $({{\sigma}^{2}}{{\varphi}_{5}}){{\sigma}^{2}}$.\\
We know that $Out\;C_{2}^{3} \simeq \stab(\lambda)$. We must prove $\stab(\lambda)=\textsl{S}_{4}$. In fact, we first denote $\varphi_{2},\varphi_{3},({\sigma}{{\varphi}_{3}}){{\sigma}^{2}}  $, respectively, by $\sigma_{1}, \sigma_{2},\sigma_{3}$ and after, with direct calculations we prove that $\varphi_{2}^{2}=\varphi_{3}^{2}=({\sigma}{{\varphi}_{3}}){{\sigma}^{2}}=~id$ and $\sigma_{1}\sigma_{2}\sigma_{1}=\sigma_{2}\sigma_{1}\sigma_{2}$, $\sigma_{2}\sigma_{3}\sigma_{2}=\sigma_{3}\sigma_{2}\sigma_{3}$, $\sigma_{1}\sigma_{3}=\sigma_{3}\sigma_{1}$. Therefore, $\stab(\lambda)=GR\{\sigma_{1},\sigma_{2}, \sigma_{3} |\sigma_{i}^{2}=1,\sigma_{i}\sigma_{i+1}\sigma_{i}=\sigma_{i+1}\sigma_{i}\sigma_{i+1}, \sigma_{i}\sigma_{j}=\sigma_{j}\sigma_{i}, j\neq i \pm 1, i=1,2,3 \}=\textsl{S}_{4} .$

\item  Let $a,b,c$ be generators of the code $ L = C_{3}^{3}$ with characteristic vector \linebreak $\lambda = \lambda(L) = (000111)$. We have that $\stab(\lambda)$ is the group of all permutations of the set $X = \left\{a,b,c,abc\right\}$. We just consider $\varphi \in GL_{3}(2)$ such that $\varphi(a)=x$, $\varphi(b)=y$ and $\varphi(c)=z$. We have to find $(x,y,z)$ such that $x^{2}=y^{2}=z^{2}=1$ and $[x,y]=[x,z]=[y,z]=-1$. From the Tables \ref{tabqccaso3} and \ref{tabqccaso3-2}, we obtain all possible values of $x, y, z$ satisfying the relations desired. Thus, $x,y,z \in \{a,b,c,abc\}.$ We consider $\sigma_{i}\in GL_{3}(2)$, for $i=1,2,3$, such that  $\sigma_{1}(a)=a$, $\sigma_{1}(b)=b$, $\sigma_{1}(c)=abc$; $\sigma_{2}(a)=a$, $\sigma_{2}(b)=abc$, $\sigma_{2}(c)=c$; $\sigma_{3}(a)=b$, $\sigma_{3}(b)=a$, $\sigma_{3}(c)=c$.\\
Analogously to the previous case, we see that $\stab(\lambda)=GR\{\sigma_{1},\sigma_{2}, \sigma_{3} |\sigma_{i}^{2}=1,\sigma_{i}\sigma_{i+1}\sigma_{i}=\sigma_{i+1}\sigma_{i}\sigma_{i+1}, \sigma_{i}\sigma_{j}=\sigma_{j}\sigma_{i}, j\neq i \pm 1, i=1,2,3 \}=\textsl{S}_{4}.$


\item Let $a,b,c$ be generators of the code loop $ L = C_{4}^{3}$ such that $\lambda = \lambda(L) = (110000)$. Since there are $21$ characteristic vectors in $O_{4}$, we have $8$ automorphisms fixing $\lambda$. To prove that $Out\; L \simeq \textsl{D}_{8}$, we must first find two generators, $\varphi$ and $\sigma$, $\varphi$ of order $4$ and $\sigma$ of order $2$,  fixing $\lambda$, and such that $\sigma \varphi \sigma = \varphi^{3}$. We consider $\varphi \in GL_{3}(2)$ and $\sigma \in GL_{3}(2)$ such that $\varphi(a)=b$, $\varphi(b)=bc$, $\varphi(c)=ab$; $\sigma(a)=b$, $\sigma(b)=a$ and $\sigma(c)=c$.
 We have that $\varphi$ and $\sigma$ fix $\lambda$. Besides, $\varphi$ has order 4, $\sigma$ has order 2 and the relations ${\sigma}{\varphi}^{3}={\varphi}{\sigma}$, ${\varphi}^{2}{\sigma}={\sigma}{\varphi}^{2}$ and ${\varphi}^{3}{\sigma}={\sigma}{\varphi}$ are valid .\\
 Therefore, $Out\;{C_{4}^{3}} \simeq \textsl{D}_{8} = GR\left\{\sigma , \varphi | {\varphi}^{4} = 1, {\sigma}^{2} = 1, {\sigma}{\varphi}{{\sigma}^{-1}} = {\varphi}^{-1} \right\}=\stab(\lambda)$.

\item Let $a,b,c$ be generators of code loop $ L = C_{5}^{3}$ such that $\lambda = \lambda(L) = (100000)$. Since we have $28$ different characteristic vectors in the orbit $O_{5}$, we have to find $6$ automorphisms fixing $\lambda$. Let $\varphi \in GL_{3}(2)$ and $\sigma \in GL_{3}(2)$ such that $\varphi(a)=a$, $\varphi(b)=c$, $\varphi(c)=b$; $\sigma(a)=ac$, $\sigma(b)=bc$ and $\sigma(c)=c$. We have $\varphi$ and $\sigma$ fix $\lambda$, $\sigma^{2}=\varphi^{2}=1$ and $\varphi \sigma \varphi = \sigma \varphi \sigma$.\\
Therefore, $Out\;C_{5}^{3} \simeq \textsl{S}_{3}=GR\{\varphi, \sigma | \varphi \sigma \varphi = \sigma \varphi \sigma\}=\stab(\lambda)$. 
 
\end{enumerate}

\end{proof}

\subsection{Code loops of rank 4}

\begin{lemma}
Let $V$ be a nontrivial $3$-space of dimension $4$. Then $dim\;N(V)=1$ and $V$ has a base $\left\{a,b,c,d\right\}$ such that $(a,b,c)=1$ and $N(V)={\bf F}_{2}d = \left\{0,d\right\}$.
\end{lemma}

\begin{proof}
Since $V$ is a nontrivial $3$-space, let $(v,w,u)=1$, for some $v,w,u \in V$. So there is $t \in V$ such that $(v,w,t) = 0$ and $\left\{v,w,u,t\right\}$ is a basis of $V$. If $(v,w,t) = 1$, then $(v,w,u+t) = (v,w,u)+(v,w,t) = 0$, soon we can assume $(v,w,t)=0.$

If $(v,u,t)=1$, then $(v,w,t+w)=(v,u,t+w)=0$ and we can assume that $(v,u,t)=0.$ If $(w,u,t)=1$, then $t+v \in N(V),$ thus we can suppose $(w,u,t)=0.$

Hence, we determine a basis $\left\{v,w,u,t\right\}$ such that $(v,w,u)=1$, $(v,w,t)$$=(v,u,t)=$$(w,u,t)=0.$

\end{proof}

 Suppose that $X=\left\{a,b,c,d\right\}$ is a basis of a nonassociative code loop $L$ of rank $4$ such that $N(L)=~\textsl{F}_{2}d$. Thus,  $L$ has only one nontrivial associator $(a,b,c)=-1$. In this case, a characteristic vector of $L$ is $\lambda_{X}=(\lambda_{1},\dots,\lambda_{10}),$ where $$a^{2}=(-1)^{\lambda_1}, b^{2}=~(-1)^{\lambda_2},c^{2}=(-1)^{\lambda_3}, d^{2}=(-1)^{\lambda_4},$$
$$ \left[a,b\right]=(-1)^{\lambda_5}, \left[a,c\right]=~(-1)^{\lambda_6}, \left[a,d\right]=(-1)^{\lambda_7},$$
$$\left[b,c\right]=~(-1)^{\lambda_8}, \left[b,d\right]=(-1)^{\lambda_9},\left[c,d\right]=(-1)^{\lambda_{10}}.$$

 \begin{theorem}\label{theorem3.5}Consider $C_{1}^{4},...,C_{16}^{4}$ the code loops with the following characteristic vectors.
All nonassociative code loop of rank $4$ is isomorphic to one from the list. Moreover, none of those loops are isomorphic to each other.
\begin{table}[!hhh]
\centering
{\begin{tabular}{llll|lllll|l}
 &  &  & $L$ & $\lambda(L)$ &  &  &  & $L$ & $\lambda(L)$ \\ 
\cline{4-5}\cline{9-10}
 &  &  & $C^{4}_{1}$ & $(1110110100)$ &  &  &  & $C^{4}_{9}$ & $(0100001000)$ \\ 
 &  &  & $C^{4}_{2}$ & $(0000000000)$ &  &  &  & $C^{4}_{10}$ & $(0001111000)$ \\ 
 &  &  & $C^{4}_{3}$ & $(0000110100)$ &  &  &  & $C^{4}_{11}$ & $(0001001000)$ \\ 
 &  &  & $C^{4}_{4}$ & $(0010100000)$ &  &  &  & $C^{4}_{12}$ & $(0000001100)$ \\ 
 &  &  & $C^{4}_{5}$ & $(0000010100)$ &  &  &  & $C^{4}_{13}$ & $(0110111100)$ \\ 
 &  &  & $C^{4}_{6}$ & $(1111110100)$ &  &  &  & $C^{4}_{14}$ & $(0001001100)$ \\ 
 &  &  & $C^{4}_{7}$ & $(0001000000)$ &  &  &  & $C^{4}_{15}$ & $(1001001100)$ \\ 
 &  &  & $C^{4}_{8}$ & $(0000001000)$ &  &  &  & $C^{4}_{16}$ & $(0001111100)$ \\ 
\end{tabular}}
\label{tabvc4}
\end{table}

Moreover, we have, if we denote by $O_i^4$ the orbit of the characteristic vector of the code loop $C^4_i$, then
$|O_{1}^{4}|=1$, $|O_{2}^{4}|=7$, $|O_{4}^{4}|=21$, $|O_{5}^{4}|=|O_{14}^{4}|=|O_{15}^{4}|=28$, $|O_{6}^{4}|=8$, $|O_{7}^{4}|=|O_{8}^{4}|=|O_{13}^{4}|=56$, $|O_{9}^{4}|=|O_{12}^{4}|=|O_{16}^{4}|=168$, $|O_{10}^{4}|=|O_{11}^{4}|=112$.

\end{theorem}

\begin{proof}
Let $L$ be a nonassociative code loop of rank $4$ with generators $a,b,c,d.$ We can assume $(d,L,L)=1$. Thus,  $L$ has only one nontrivial associator $(a,b,c)=-1$. In this case, a characteristic vector of $L$ is $\lambda_{X}=(\lambda_{1},\dots,\lambda_{10}),$ where $$a^{2}=(-1)^{\lambda_1}, b^{2}=~(-1)^{\lambda_2},c^{2}=(-1)^{\lambda_3}, d^{2}=(-1)^{\lambda_4},$$
$$ \left[a,b\right]=(-1)^{\lambda_5}, \left[a,c\right]=~(-1)^{\lambda_6}, \left[a,d\right]=(-1)^{\lambda_7},$$
$$\left[b,c\right]=~(-1)^{\lambda_8}, \left[b,d\right]=(-1)^{\lambda_9},\left[c,d\right]=(-1)^{\lambda_{10}}.$$

 Let $C=C(d)$ be the centralizer of $d$, that is, $C=\left\{v \in L | [v,d]=1\right\}$.

A subloop $L_{0}$ of $L$ is a characteristic subloop if and only if $ L_{0}^{\varphi} = L_{0}$, for all $\varphi \in Aut(L)$. Then, we have that $C$ is a characteristic subloop of $L$.

If $C=L$, then we have two possibilities:

\begin{enumerate}
\item $d^{2}=1$\\
In this case, $L\simeq \mathbb{Z}_{2}\times L_{1}$, where $L_{1}$ is a nonassociative code loop of rank $3$.  Since there are $5$ nonassociatives code loops of rank $3$, then we obtain $5$ non-isomorphic code loops of rank $4$, denoted by $C_{i}^{4}$, that is, $C_{i}^{4} \cong \textsl{Z}_{2}{\times}C_{i}^{3}, i = 1,\dots,5.$

\item $d^{2}=-1$\\
Let $\left\langle d\right\rangle = \left\{1,-1,d,-d\right\}$ be the group gerated by $d$. We know that $\left\langle d\right\rangle \cong \textsl{Z}_{4}$ and that $\textsl{Z}_{4}$ is code loop. Again, let $L_{1}$ be a nonassociative code loop of rank $3$. We have that $\textsl{Z}_{4}{*}L_{1}$ is a nonassociative code loop of rank $4$.

In this case, $L\simeq \textsl{Z}_{4}{*}L_{1}$. But, by the Lemma \ref{lemma2.15}, it is possible to prove that 
$\textsl{Z}_{4}{*}C_{1}^{3} \cong \textsl{Z}_{4}{*}C_{3}^{3}$ and $\textsl{Z}_{4}{*}C_{2}^{3} \cong \textsl{Z}_{4}{*}C_{4}^{3} \cong \textsl{Z}_{4}{*}C_{5}^{3}.$ For this, we just prove that $C_{1}^{3}$ and $C_{3}^{3}$ are isomorphic as $2$-spaces, as well as $C_{2}^{3}$, $C_{4}^{3}$ and $C_{5}^{3}$. In fact, a characteristic vector from corresponding $2$-space of a given code loop $L$ of rank $n$ can be obtained from characteristic vector $\lambda(L)$ by omitting the first $n$ coordenates of $\lambda(L)$. Therefore, we have 2 news loops: $C_{6}^{4}=\textsl{Z}_{4}*C_{1}^{3}$ and $C_{7}^{4}=\textsl{Z}_{4}*C_{2}^{3}$.
 \end{enumerate}

If $C \neq L,$ then $C$ is a code group of rank $3$. In fact, we suppose that for all $u \neq d \in X = \left\{a,b,c,d\right\}$, the set of generators of $L$, we have $[u,d]=-1$, but since $[ab,d]=[ac,d]=1$, hence we have that $ \left\{d,ab,ac\right\}$ generates $C$. Now, if we assume that, for example, $[a,d]=1$,$[b,d]=[c,d]=-1$, we will have $[bc,d]=1$, hence $ \left\{d,a,bc\right\}$ generates $C$.

There are $5$ non-isomorphic code groups of rank $3$:
\vspace{-0.1in}
\begin{enumerate}
\item  $G_{1}^{3}=\textsl{Z}_{2}^{4}$;
\item  $G_{2}^{3}=\textsl{Z}_{4}{\times}\textsl{Z}_{2}^{2}$;
\item  $G_{3}^{3}=\textsl{D}_{8}{\times}\textsl{Z}_{2}$;
\item  $G_{4}^{3}=\textsl{Q}{\times}\textsl{Z}_{2}$, where $Q$ is the Quaternion group;
\item  $G_{5}^{3}=\textsl{D}_{8}{*}\textsl{Z}_{4}$, where $\textsl{D}_{8}$ is the dihedral group with $8$ elements.
\end{enumerate}

We are going to do an analysis for each case:
\begin{enumerate}

\item  $C=G_{1}^{3}$. In this case, $b,c,d \in C$, then $b^{2}=c^{2}=d^{2}=[b,c]=[b,d]=[c,d]=1$ and $[a,d]=-1$. If $a^{2}=-1$, then $(ad)^{2}=a^{2}d^{2}[a,d]=1.$ Hence, we can assume $a^{2}=1.$ We denote this loop by $C_{8}^{4}.$

\item  $C=G_{2}^{3}$. There are $2$ cases:

        (a)\; $d^{2}=1$. Here, $b^{2}=-1$, $c^{2}=1$, $[a,d]=-1$, $[a,b]=[a,c]=1$. As in the previous case, we can assume $a^{2}=1.$ We denote this loop by $C_{9}^{4}$.

        (b)\; $d^{2}=-1$. We have $b^{2}=c^{2}=1$, $[a,d]=-1$. If $[a,b]=[a,c]=-1$, then $[a,bc]=-1$, $(ab)^{2}=-a^{2}$ and hence, we can assume $a^{2}=1.$ We denote this loop by $C_{10}^{4}$. If $[a,b]=1$, then without loss of generality, $[a,c]=1$, since that $[a,bc]=-[a,c].$ Since $(ac)^{2}=-a^{2}$, we can assume $a^{2}=1.$ We denote this loop by $C_{11}^{4}$.

\item  $C=G_{3}^{3}$. In this case, $d^{2}=1$, $c^{2}=b^{2}=1$, $[b,c]=-1$, $[a,d]=-1$. Since $[a,bd]=-[a,b],$ $(ad)^{2}=-a^{2},$ we can assume $[a,b]=[a,c]=1$ and $a^{2}=1.$ We denote this loop by $C_{12}^{4}$.

\item   $C=G_{4}^{3}$. In this case, $d^{2}=1$, $a^{2}=1$, $b^{2}=c^{2}=[b,c]=-1$, $[a,d]=-1$. There are $2$ possibilities. If $[a,b]=[a,c]=-1$, we denote this loop by $C_{13}^{4}$.\\ If $[a,b]=[a,c]=1$, this loop is isomorphic to $C_{13}^{4}$, since we can replace $b$ by $bc$ and $c$ by $cd$, because $(bc)^{2}=(cd)^{2}=[bc,cd]=-1$, $[a,bc]=[a,cd]=-1$.

\item  $C=G_{5}^{3}$. In this case, $d^{2}=-1$, $a^{2}=1$, $b^{2}=c^{2}=1$, $[b,c]=-1$, $[a,d]=-1$. There are $3$ possibilities:

        (a)\; $[a,b]=[a,c]=1$, $[a,bc]=-1$. In this subcase, we have two non-isomorphic loops. If $a^{2}=1$, we have the loop denoted by $C_{14}^{4}$, and if $a^{2}=-1$ then we have the other loop denoted by $C_{15}^{4}$.
                
        (b)\; $[a,b]=[a,c]=-1$, $[a,bc]=-1$. In this other subcase, $(abd)^{2}=-a^{2}$ and $[abd,b]=[abd,c]=-1$. Therefore, we can assume $a^{2}=1$. We denote this loop by $C_{16}^{4}$.
        
        (c)\; $[a,b]=1$, $[a,c]=-1$, $[a,bc]=1$. This loop is isomorphic to loop $C_{16}^{4}$, since we can replace $b$ by $bcd$.
\end{enumerate}
  
\end{proof}

\begin{corollary}
There are five $\mathbb{Z}_{4}$-equivalence classes of code loops of rank $4$:
\begin{eqnarray*}
\begin{array}{llll}
Z_{1}^{4}=\left\{C_{1}^{4},C_{3}^{4},C_{4}^{4},C_{6}^{4}\right\} &&&
Z_{2}^{4}=\left\{C_{5}^{4},C_{8}^{4},C_{9}^{4},C_{10}^{4},C_{11}^{4}\right\}\\
Z_{3}^{4}=\left\{C_{12}^{4},C_{14}^{4},C_{15}^{4}\right\}&&&
Z_{4}^{4}=\left\{C_{2}^{4},C_{7}^{4}\right\}\\
Z_{5}^{4}=\left\{C_{13}^{4},C_{16}^{4}\right\}&&& \\
\end{array}
\end{eqnarray*}
\end{corollary}

\begin{proof}
From the Lemma \ref{lemma2.15}, we have to prove that $C_{i}^{4}\cong C_{j}^{4}$ as $2$-spaces if and only if $C_{i}^{4},C_{j}^{4} \in Z_{k}^{4}$ for  some $k=1,\dots,4.$ We will write $i \sim j$ instead $C_{i}^{4} {\sim}_{\Z_{4}} C_{j}^{4}$. 

A characteristic vector of a $2$-space which corresponds to a given code loop $L$ of rank $n$, can be obtained from the characteristic vector $\lambda(L)$ by omitting the first $n$ coordinates of $\lambda(L)$.

Therefore, from the Theorem \ref{theorem3.5}, we have $1 \sim 3 \sim 6 $, $2 \sim 7$, $8 \sim 9 \sim 11$, $12 \sim 14 \sim 15$ and $13\sim 16$.

Let $X=\left\{a,b,c,d\right\}$ be a set of generators of $C_{3}^{4}$ such that $(a,b,c)=[a,b]=[a,c]=[b,c]=-1$, but then $[a,abc]=[b,abc]=1$ and, for a set of generators $Y=\left\{a,b,abc,d\right\}$, we have a characteristic vector from the $2$-space $C_{3}^{4}$ of the form: $\lambda_{Y}(C_{3}^{4})=(100000)=\lambda(C_{4}^{4}).$ Therefore, $3 \sim 4.$

Let $X=\left\{a,b,c,d\right\}$ be a set of generators of $C_{5}^{4}$ such that $(a,b,c)=[a,c]=[a,d]=-1$, but then $[a,cd]=1$ and for a set of generators $Y=\left\{a,b,cd,d\right\}$ we have a characteristic vector of the $2$-space $C_{5}^{4}$ of the form: $\lambda_{Y}(C_{5}^{4})=(001000)=\lambda(C_{8}^{4}).$ Therefore, $5 \sim 8.$

Let $X=\left\{a,b,c,d\right\}$ a set of generators of $C_{10}^{4}$ such that $(a,b,c)=[a,b]=[a,c]=[a,d]=-1$, but then $[a,bd]=[a,cd]=1$ and, for a set of generators $Y=\left\{a,bd,cd,d\right\}$, we have a characteristic vector of the $2$-space $C_{10}^{4}$ of the form: $\lambda_{Y}(C_{10}^{4})=(001000)=\lambda(C_{8}^{4}).$ Therefore, $10 \sim 8.$

Now, we have to prove that all these classes $Z_{1}^{4},\dots ,Z_{5}^{4}$ are nonequivalent. A space $C(N(L))=\{x \in L | [x,N(L)]=1\}$ is a characteristic space of $L$. In our case,\linebreak $N(L)={\bf F}_{2}d$ and for $C=C(d)$ we have three possibilities:
\begin{enumerate}
\item $C=L$,
\item $[C,C]=(C,C,C)=1$,
\item $C\neq L,$ $[C,C]\neq 1$.
\end{enumerate}
Besides, we have the second possibility only for $Z_{3}^{4}$ and the third possibility only for $Z_{4}^{4}$. To prove that $1\not\sim 2$ we note that  $N_{2}(C_{1}^{4})\simeq N_{2}(C_{2}^{4})\simeq {\bf F}_{2}d$, but the  $2$-space $C_{1}^{4}/N_{2}(C_{1}^{4})$ and $C_{2}^{4}/N_{2}(C_{2}^{4})$ are non-isomorphic. Analogously, we prove that $1\not\sim 5$ and $2\not\sim 5$.

\end{proof}

 Let $(a,b,c,d)$ be a set of generators of the code loop of rank $4$ denoted by $C_{i}^{4}$ and $\sigma \in GL_{4}(2)$ an automorphism of $ C_{i}^{4}$. For each $\sigma$ we find a new set of generators of the same code loop, denoted by $(u,v,w,d)$, that is, $a^{\sigma}=u, b^{\sigma}=v, c^{\sigma}=w, d^{\sigma}=d$. Analogously to the study of the orbits of the code loops of rank $3$, Proposition \ref{prop3.2}, we can exhibit all the codes of each orbit associated to a given code loop of code $4$. In fact, just fix a representative characteristic vector  of a certain orbit, runs through $GL_{4}(2)$ and find all possible set of generators $(u,v,w,d)$. So from the square of each generator and the commutators of each pair of generators of $C_{i}^{4}$, we find a characteristic vector that belongs to orbit of the representative characteristic vector fixed.

When is necessary we denote the characteristic vector $\lambda(C_{i}^{k})$ by $\lambda_{i}^{k}$, where $k$ represents the rank of the code loop. We denote by $\mbox{\normalfont Stab}(\lambda_{i}^{k})$ the  group of automorphisms of $GL_{k}(2)$ that fix $\lambda_{i}^{k}$.

For the next proposition, $H \times K$ means the direct product of the groups $H$ and $K$, and $H \rtimes K$ means the semidirect product of the groups $H$ and $K$.

\begin{proposition} In the notation above, we have:

\begin{math}
 \begin{array}{llll}
        \mbox{\normalfont Out} C_{1}^{4} \simeq \mathbb{Z}_{2}^{3}\rtimes GL_{3}(2) &     &  & \mbox{\normalfont Out} C_{9}^{4} \simeq D_{8} \\ 
       \mbox{\normalfont Out} C_{2}^{4} \simeq \mathbb{Z}_{2}^{3}\rtimes S_{4} &  &     & \mbox{\normalfont Out} C_{10}^{4} \simeq S_{3}\times \mathbb{Z}_{2} \\ 
       \mbox{\normalfont Out} C_{3}^{4} \simeq \mathbb{Z}_{2}^{3}\rtimes S_{4} &  &     & \mbox{\normalfont Out} C_{11}^{4} \simeq S_{3}\times \mathbb{Z}_{2} \\ 
       \mbox{\normalfont Out} C_{4}^{4} \simeq \mathbb{Z}_{2}^{3}\rtimes D_{8} &   &   & \mbox{\normalfont Out} C_{12}^{4} \simeq D_{8} \\ 
       \mbox{\normalfont Out} C_{5}^{4} \simeq S_{4}\times \mathbb{Z}_{2} &  &  &    \mbox{\normalfont Out} C_{13}^{4} \simeq S_{4} \\ 
       \mbox{\normalfont Out} C_{6}^{4} \simeq GL_{3}(2) &  &      & \mbox{\normalfont Out} C_{14}^{4} \simeq \mathbb{Z}_{2}^{3}\rtimes S_{3} \\ 
       \mbox{\normalfont Out} C_{7}^{4} \simeq S_{4} &  &  &     \mbox{\normalfont Out} C_{15}^{4} \simeq \mathbb{Z}_{2}^{3}\rtimes S_{3} \\ 
       \mbox{\normalfont Out} C_{8}^{4} \simeq S_{4} &  &  &     \mbox{\normalfont Out} C_{16}^{4} \simeq \mathbb{Z}_{2}^{3} \\ 
\end{array}
\end{math}
\end{proposition}

\begin{proof}
\begin{enumerate}


\item Let $(a,b,c,d)$ be a basis of the loop $C_{1}^{4}$ with characteristic  vector \linebreak $\lambda_{1}^{4}=(1110110100)$, hence we have $a^{2}=b^{2}=c^{2}=d^{2}=-1$ and $[a,b]=[a,c]=[a,d]=[b,c]=[b,d]=[c,d]=-1$. We also have that $(a,b,c)$ is a basis of the code loop of rank $3$ with characteristic vector $\lambda_{1}^{3}=(111111)$.

Let $\varphi \in \stab(\lambda_{1}^{3})$. We extend $\varphi$ to $\tilde{\varphi}$ defining: \begin{eqnarray*}
\begin{array}{llll}
\tilde{\varphi}: & a &  & a^{\varphi}d^{\alpha_{1}}=u \\ 
 & b & \longmapsto & b^{\varphi}d^{\alpha_{2}}=v \\ 
 & c &  & c^{\varphi}d^{\alpha_{3}}=w \\ 
 & d &  & d \\ 
\end{array}, \;\;\alpha_{i}\in \{0,1\}, i=1,2,3.\end{eqnarray*}

The matrix which represents $\tilde{\varphi}$ is of the form 
\begin{eqnarray*}
\tilde{\varphi}=\left(\begin{array}{lll|l}
 &  &  & \alpha_{1} \\ 
\multicolumn{3}{c|}{\varphi} & \alpha_{2} \\ 
 &  &  & \alpha_{3} \\ 
\hline
0 & 0 & 0 & 1 \\ 
\end{array}\right).
\end{eqnarray*}
Clearly $\tilde{\varphi} \in GL_{4}(2)$, so it remains to prove that $\tilde{\varphi} \in \stab(\lambda_{1}^{4})$, that is,   $(\lambda_{1}^{4})^{\tilde{\varphi}}=\lambda_{1}^{4}.$ For this, we should find for which $(\alpha_{1},\alpha_{2},\alpha_{3})$ we have the following relations:\\
(a)\; $u^{2}=v^{2}=w^{2}=1$\\
(b)\;$[u,v]=[u,w]=[v,w]=-1$ and $[u,d]=[v,d]=[w,d]=1$

Let's see:

(a)\; Case $\alpha_{i}=0$: $(a^{\varphi})^{2}=(b^{\varphi})^{2}=(c^{\varphi})^{2}=-1$, since $\varphi \in \stab(\lambda_{1}^{3})$.\\
      Case $\alpha_{i}=1$: $(a^{\varphi}d)^{2}=(a^{\varphi})^{2}d^{2}[a^{\varphi},d]=-1$, because $(a^{\varphi})^{2}=-1$ ($\varphi \in \stab(\lambda_{1}^{3})$) and $d^{2}=[a^{\varphi},d]=1$ $([a,d]=[b,d]=[c,d]=1)$.

(b)\; Clearly, we have $[u,d]=[v,d]=[w,d]=1$. Now, we will see that\linebreak $[u,v]=[u,w]=[v,w]=-1$. If $\alpha_{i}=0$, then we have $[a^{\varphi},b^{\varphi}]=-1$, because $\varphi \in \stab(\lambda_{1}^{3})$. If $\alpha_{i}=1$, we have $[a^{\varphi}d,b^{\varphi}d]=[a^{\varphi},b^{\varphi}d][d,b^{\varphi}d](a^{\varphi},b^{\varphi},d)$, but $[d,b^{\varphi}d]=(a^{\varphi},b^{\varphi},d)=1$, then $[a^{\varphi}d,b^{\varphi}d]=[a^{\varphi},b^{\varphi}][a^{\varphi},d]=[a^{\varphi},b^{\varphi}]=-1$. Hence, $[u,v]=-1$. Analogously, we prove the other two cases.      

Therefore, for all $\varphi \in \stab(\lambda_{1}^{3})$ and for any values of $(\alpha_{1},\alpha_{2},\alpha_{3})$, we have $\tilde{\varphi} \in \stab(\lambda_{1}^{4})$. By the Proposition \ref{prop3.3} (groups of automorphisms of code loops of rank $3$), we have $|\stab(\lambda_{1}^{3})|=168$.
Hence $|\stab(\lambda_{1}^{4})|=168.8=1344$.

We consider the following stabilizers of $\lambda_{1}^{4}$:
{
\footnotesize
\[\begin{array}{lll}
\sigma_{1}=\left(\begin{array}{llll}
1 & 0 & 0 & 0 \\ 
1 & 1 & 0 & 0 \\ 
0 & 0 & 1 & 0 \\ 
0 & 0 & 0 & 1 \\ 
\end{array}\right)\;\;&\;\; \sigma_{2}=\left(\begin{array}{llll}
1 & 0 & 0 & 0 \\ 
0 & 1 & 0 & 0 \\ 
1 & 0 & 1 & 0 \\ 
0 & 0 & 0 & 1 \\ 
\end{array}\right)\;\; &\;\; \sigma_{3}=\left(\begin{array}{llll}
1 & 0 & 0 & 0 \\ 
0 & 1 & 0 & 0 \\ 
0 & 0 & 1 & 1 \\ 
0 & 0 & 0 & 1 \\ 
\end{array}\right) \\ 
 
\end{array}\]}
{
\footnotesize
\[\begin{array}{lll}
\sigma_{4}=\left(\begin{array}{llll}
1 & 0 & 0 & 0 \\ 
0 & 1 & 1 & 0 \\ 
0 & 0 & 1 & 0 \\ 
0 & 0 & 0 & 1 \\ 
\end{array}\right)\;\;&\;\; \sigma_{5}=\left(\begin{array}{llll}
1 & 1 & 0 & 0 \\ 
0 & 1 & 0 & 0 \\ 
0 & 0 & 1 & 0 \\ 
0 & 0 & 0 & 1 \\ 
\end{array}\right)\;\; &\;\; \sigma_{6}=\left(\begin{array}{llll}
1 & 0 & 1 & 0 \\ 
0 & 1 & 0 & 0 \\ 
0 & 0 & 1 & 0 \\ 
0 & 0 & 0 & 1 \\ 
\end{array}\right) \\ 
 
\end{array}\]}
{
\footnotesize
\[\begin{array}{lll}
\rho_{1}=\left(\begin{array}{llll}
1 & 0 & 0 & 0 \\ 
0 & 1 & 0 & 1 \\ 
0 & 0 & 1 & 0 \\ 
0 & 0 & 0 & 1 \\ 
\end{array}\right)\;\;&\;\; \rho_{2}=\left(\begin{array}{llll}
1 & 0 & 0 & 0 \\ 
0 & 1 & 0 & 1 \\ 
0 & 0 & 1 & 1 \\ 
0 & 0 & 0 & 1 \\ 
\end{array}\right)\;\; &\;\; \rho_{3}=\left(\begin{array}{llll}
1 & 0 & 0 & 1 \\ 
0 & 1 & 0 & 0 \\ 
0 & 0 & 1 & 0 \\ 
0 & 0 & 0 & 1 \\ 
\end{array}\right) \\ 
 
\end{array}\]}

We have $\sigma_{i}^{2}=2$, $i=1,\dots,6$, $\rho_{i}^{2}=1$, $\rho_{i}\rho_{j}=\rho_{j}\rho_{i}$, $i,j=1,2,3$. Let \linebreak $H=GR\{\sigma_{i}, i=1,\dots,6\}$ and $N=GR\{\rho_{i},i=1,2,3|\rho_{i}^{2}=1, \rho_{i}\rho_{j}=\rho_{j}\rho_{i}\}$, be subgroups of $\stab(\lambda_{1}^{4})$. We know that $H\simeq GL_{3}(2)$, $N \simeq \Z_{2}^{3}$, and we still have $\stab(\lambda_{1}^{4})=HN$ and $H\cap N=\{1\}$. Besides, with direct calculations we prove that $N$ is normal in $\stab(\lambda_{1}^{4})$. Therefore, $\Out(C_{1}^{4})\simeq \Z_{2}^{3}\rtimes GL_{3}(2)$.

\item In this case $\lambda_{2}^{4}=(0000000000)$, that is, $a^{2}=b^{2}=c^{2}=d^{2}=1$ and $[a,b]=[a,c]=[a,d]=[b,c]=[b,d]=[c,d]=1$. The basis $(a,b,c)$ give us the code loop of rank $3$, with $\lambda_{2}^{3}=(000000)$. 

Let $\varphi \in  \stab(\lambda_{2}^{3})$. We extend $\varphi$ to $\tilde{\varphi} \in GL_{4}(2)$, as in the previous case. We must prove that $\tilde{\varphi} \in \stab(\lambda_{2}^{4})$, that is, $(\lambda_{2}^{4})^{\tilde{\varphi}}=\lambda_{2}^{4}.$ In fact, since  $\varphi \in  \stab(\lambda_{2}^{3})$, then

(a)\;  $(a^{\varphi}d^{\alpha_{1}})^{2}=(a^{\varphi})^{2}(d^{\alpha_{1}})^{2}[a^{\varphi},d^{\alpha_{1}}]=1$. Analogously, $(b^{\varphi}d^{\alpha_{2}})^{2}=(c^{\varphi}d^{\alpha_{3}})^{2}=1$.

(b)\; Clearly, $[a^{\varphi}d^{\alpha_{1}},d]=[b^{\varphi}d^{\alpha_{2}},d]=[c^{\varphi}d^{\alpha_{3}},d]=1$.  

Now, we have that  $[a^{\varphi}d^{\alpha_{1}},b^{\varphi}d^{\alpha_{2}}]=[a^{\varphi},b^{\varphi}d^{\alpha_{2}}][d^{\alpha_{1}},b^{\varphi}d^{\alpha_{2}}](a^{\varphi},d^{\alpha_{1}},b^{\varphi}d^{\alpha_{2}})$. Since \linebreak  $[d^{\alpha_{1}},b^{\varphi}d^{\alpha_{2}}]=(a^{\varphi},d^{\alpha_{1}},b^{\varphi}d^{\alpha_{2}}) =1$, and $[a^{\varphi},b^{\varphi}]=[a^{\varphi},d^{\alpha_{2}}]=(a^{\varphi},b^{\varphi},d^{\alpha_{2}})=1$, we obtain $[a^{\varphi}d^{\alpha_{1}},b^{\varphi}d^{\alpha_{2}}]=1$. Similarly,  $[a^{\varphi}d^{\alpha_{1}},c^{\varphi}d^{\alpha_{3}}]=[b^{\varphi}d^{\alpha_{2}},c^{\varphi}d^{\alpha_{3}}]=1$.

We proved that, $\tilde{\varphi} \in \stab(\lambda_{2}^{4})$, for all $\varphi \in \stab(\lambda_{2}^{3})$ and for any values of $(\alpha_{1},\alpha_{2},\alpha_{3})$.

By the Proposition \ref{prop3.3}, we have $|\stab(\lambda_{2}^{3})|=24$.
Hence, $|\stab(\lambda_{2}^{4})|=24.8=192$. 

To prove that $\Out C_{2}^{4} \simeq \Z_{2}^{3}\rtimes S_{4}$, we just need to consider the subgroups $H$ and $N$ of $\stab(\lambda_{2}^{4})$ given respectively by: $$H=GR\{\sigma_{i},i=1,2,3| \sigma_{i}^{2}=1, \sigma_{i}\sigma_{i+1}\sigma_{i}=\sigma_{i+1}\sigma_{i}\sigma_{i+1}, \sigma_{1}\sigma_{3}=\sigma_{3}\sigma_{1}\}\simeq S_{4}$$ and $$N=GR\{\rho_{i},i=1,2,3|\rho_{i}^{2}=1, \rho_{i}\rho_{j}=\rho_{j}\rho_{i}\}\simeq \Z_{2}^{3}.$$

In this case, we consider the generators these groups given as follows:
{
\footnotesize
\[\begin{array}{lll}
\sigma_{1}=\left(\begin{array}{llll}
1 & 0 & 0 & 0 \\ 
0 & 0 & 1 & 0 \\ 
0 & 1 & 0 & 0 \\ 
0 & 0 & 0 & 1 \\ 
\end{array}\right)\;\;&\;\; \sigma_{2}=\left(\begin{array}{llll}
0 & 1 & 0 & 0 \\ 
1 & 0 & 0 & 0 \\ 
0 & 0 & 1 & 0 \\ 
0 & 0 & 0 & 1 \\ 
\end{array}\right)\;\; &\;\; \sigma_{3}=\left(\begin{array}{llll}
1 & 0 & 0 & 0 \\ 
1 & 1 & 0 & 0 \\ 
1 & 0 & 1 & 0 \\ 
0 & 0 & 0 & 1 \\ 
\end{array}\right) \\ 
 
\end{array}\]}
{
\footnotesize
\[\begin{array}{lll}
\rho_{1}=\left(\begin{array}{llll}
1 & 0 & 0 & 0 \\ 
0 & 1 & 0 & 1 \\ 
0 & 0 & 1 & 0 \\ 
0 & 0 & 0 & 1 \\ 
\end{array}\right)\;\;&\;\; \rho_{2}=\left(\begin{array}{llll}
1 & 0 & 0 & 0 \\ 
0 & 1 & 0 & 1 \\ 
0 & 0 & 1 & 1 \\ 
0 & 0 & 0 & 1 \\ 
\end{array}\right)\;\; &\;\; \rho_{3}=\left(\begin{array}{llll}
1 & 0 & 0 & 1 \\ 
0 & 1 & 0 & 0 \\ 
0 & 0 & 1 & 0 \\ 
0 & 0 & 0 & 1 \\ 
\end{array}\right) \\ 
 
\end{array}\]}
We have that $\stab(\lambda_{2}^{4})=HN$, $H\cap N=\{1\}$ and $N$ is a normal subgroup of $\stab(\lambda_{2}^{4})$.


\item In this case $\lambda_{3}^{4}=(0000110100)$ and $\lambda_{3}^{3}=(000111)$. 

We have $(a^{\tilde{\varphi}})^{2}=(a^{\varphi}d^{\alpha_{1}})^{2}=(a^{\varphi})^{2}(d^{\alpha_{1}})^{2}[a^{\varphi},d]=1.$ Analogously, we have \linebreak  $(b^{\tilde{\varphi}})^{2}=(c^{\tilde{\varphi}})^{2}=1$. Then, for any triple $(\alpha_{1},\alpha_{2},\alpha_{3})$, we have $\tilde{\varphi} \in \stab(\lambda_{3}^{4})$. Besides, we know that $|\stab(\lambda_{3}^{3})|=24$, hence  $|\stab(\lambda_{3}^{4})|=24.8=192$.

Analogously to previous case, we prove that  $\Out C_{2}^{4} \simeq \Z_{2}^{3}\rtimes S_{4}$. In this case, we consider the generators of the groups $H\simeq S_{4}$ and $N\simeq \Z_{2}^{3}$, given by:
{
\footnotesize
\[\begin{array}{lll}
\sigma_{1}=\left(\begin{array}{llll}
1 & 0 & 0 & 0 \\ 
0 & 0 & 1 & 0 \\ 
0 & 1 & 0 & 0 \\ 
0 & 0 & 0 & 1 \\ 
\end{array}\right)\;\;&\;\; \sigma_{2}=\left(\begin{array}{llll}
0 & 1 & 0 & 0 \\ 
1 & 0 & 0 & 0 \\ 
0 & 0 & 1 & 0 \\ 
0 & 0 & 0 & 1 \\ 
\end{array}\right)\;\; &\;\; \sigma_{3}=\left(\begin{array}{llll}
1 & 1 & 1 & 0 \\ 
0 & 1 & 0 & 0 \\ 
0 & 0 & 1 & 0 \\ 
0 & 0 & 0 & 1 \\ 
\end{array}\right) \\ 
 
\end{array}\]}
{
\footnotesize
\[\begin{array}{lll}
\rho_{1}=\left(\begin{array}{llll}
1 & 0 & 0 & 0 \\ 
0 & 1 & 0 & 1 \\ 
0 & 0 & 1 & 0 \\ 
0 & 0 & 0 & 1 \\ 
\end{array}\right)\;\;&\;\; \rho_{2}=\left(\begin{array}{llll}
1 & 0 & 0 & 0 \\ 
0 & 1 & 0 & 1 \\ 
0 & 0 & 1 & 1 \\ 
0 & 0 & 0 & 1 \\ 
\end{array}\right)\;\; &\;\; \rho_{3}=\left(\begin{array}{llll}
1 & 0 & 0 & 1 \\ 
0 & 1 & 0 & 0 \\ 
0 & 0 & 1 & 0 \\ 
0 & 0 & 0 & 1 \\ 
\end{array}\right) \\ 
 
\end{array}\]}

\item Let $\lambda_{4}^{4}=(0010100000)$ and $\lambda_{4^{'}}^{3}=(001100)$.  We have that $|\stab(\lambda_{4^{'}}^{3})|=8$ because $\lambda_{4^{'}}^{3} \in O_{4}^{3}$. In fact, we assume $a^{\varphi}=x, b^{\varphi}=y, c^{\varphi}=z $. We have $8$ possibilities for $(x,y,z)$, given by $(a,b,c)$, $(b,a,c)$, $(a,b,ac)$, $(b,a,ac)$, $(a,b,bc)$, $(b,a,bc)$, $(a,b,abc)$, $(b,a,abc)$. In any case $x^{2}=y^{2}=1, z^{2}=-1$, $[x,y]=-1$, $[x,y]=[y,z]=1$. 
Now, we have $(a^{\tilde{\varphi}})^{2}=(a^{\varphi}d^{\alpha_{1}})^{2}=(a^{\varphi})^{2}(d^{\alpha_{1}})^{2}[a^{\varphi},d]=1$. Analogously, we have $(b^{\tilde{\varphi}})^{2}=1$ and $(c^{\tilde{\varphi}})^{2}=-1$. Then, for any triple $(\alpha_{1},\alpha_{2},\alpha_{3})$, we have $\tilde{\varphi} \in \stab(\lambda_{4}^{4})$. Hence  $|\stab(\lambda_{4}^{4})|=8.8=64$.

We consider the following stabilizers of $\lambda_{4}^{4}$:
{
\footnotesize
\[\begin{array}{lll}
\sigma_{1}=\left(\begin{array}{llll}
0 & 1 & 0 & 0 \\ 
1 & 0 & 0 & 0 \\ 
0 & 1 & 1 & 0 \\ 
0 & 0 & 0 & 1 \\ 
\end{array}\right)\;\;&\;\; \sigma_{2}=\left(\begin{array}{llll}
1 & 0 & 0 & 0 \\ 
0 & 1 & 0 & 0 \\ 
1 & 0 & 1 & 0 \\ 
0 & 0 & 0 & 1 \\ 
\end{array}\right) \\ 
 
\end{array}\]}

{
\footnotesize
\[\begin{array}{lll}
\rho_{1}=\left(\begin{array}{llll}
1 & 0 & 0 & 0 \\ 
0 & 1 & 0 & 1 \\ 
1 & 0 & 1 & 1 \\ 
0 & 0 & 0 & 1 \\ 
\end{array}\right)\;\;&\;\; \rho_{2}=\left(\begin{array}{llll}
1 & 0 & 0 & 1 \\ 
0 & 1 & 0 & 0 \\ 
0 & 1 & 1 & 1 \\ 
0 & 0 & 0 & 1 \\ 
\end{array}\right)\;\; &\;\; \rho_{3}=\left(\begin{array}{llll}
1 & 0 & 0 & 0 \\ 
0 & 1 & 0 & 0 \\ 
0 & 0 & 1 & 1 \\ 
0 & 0 & 0 & 1 \\ 
\end{array}\right) \\ 
 
\end{array}\]}

With direct calculations we obtain: $\sigma_{1}^{4}=1$, $\sigma_{2}^{2}=1$, $\sigma_{2}\sigma_{1}\sigma_{2}=\sigma_{1}^{3}$, $\rho_{1}^{2}=\rho_{2}^{2}=\rho_{3}^{2}=1$, $\rho_{i}\rho_{j}=\rho_{j}\rho_{i}$, $i,j=1,2,3$.

Let $K_{1}=GR\{\sigma_{1}, \sigma_{2}\;|\; \sigma_{1}^{4}=1, \sigma_{2}^{2}=1, \sigma_{2}\sigma_{1}\sigma_{2}=\sigma_{1}^{3}\}$ and  $H_{1}=GR\{\rho_{1}, \rho_{2}, \rho_{3}\;|\; \rho_{1}^{2}=\rho_{2}^{2}=\rho_{3}^{2}=1, \rho_{i}\rho_{j}=\rho_{j}\rho_{i},\; i,j=1,2,3\}.$
The elements these two groups are also stabilizers of $\lambda_{4}^{4}$. Indeed, the product of any element of $K_{1}$ by $H_{1}$ is a stabilizer of $\lambda_{4}^{4}$. Hence we have $H_{1}\cap K_{1}=\{1\}$ and $\stab(\lambda_{4}^{4})=K_{1}H_{1}$. 

We know that $K_{1}\simeq D_{8}$ and $H_{1}\simeq \Z_{2}^{3}$. We still have that $K_{1}$ isn't a normal subgroup of $\stab(\lambda_{4}^{4})$, because, for example, for $\rho_{1}\in H_{1}$ and $\sigma_{1} \in K_{1}$ we don't have $\rho_{1}\sigma_{1}\rho_{1} \in K_{1}$, that is, $\rho_{1}K_{1}\rho_{1}\nsubseteq K_{1}$. For the other side, $H_{1}\triangleleft \stab(\lambda_{4}^{4})$.  To prove this, we just do the calculations directly and we use the following relations:

$\sigma_{2}\rho_{1}=\rho_{1}\sigma_{2}$, 
$\sigma_{2}\rho_{2}=\rho_{2}\sigma_{2}$,
$\sigma_{1}\rho_{3}=\rho_{3}\sigma_{1}$,
$\sigma_{2}\rho_{3}=\rho_{3}\sigma_{2}$,
$\rho_{1}\sigma_{1}=\sigma_{1}\rho_{2}$,
$\sigma_{1}\rho_{1}\rho_{3}=\rho_{2}\sigma_{1}$,
$\rho_{1}\rho_{2}\sigma_{1}=\sigma_{1}\rho_{2}\rho_{1}\rho_{3}$,
$\sigma_{1}\rho_{1}=\rho_{2}\rho_{3}\sigma_{1}$,
$\sigma_{1}\rho_{2}\rho_{3}=\rho_{1}\rho_{3}\sigma_{1}$,
$\sigma_{1}\rho_{1}\rho_{2}=\rho_{2}\rho_{3}\rho_{1}\sigma_{1}$.

Therefore, $\stab(\lambda_{4}^{4}) \simeq \Z_{2}^{3}\rtimes D_{8}.$

\item Let $\lambda_{5}^{4}=(0000010100)$ and $\lambda_{5^{'}}^{3}=(000011)$. In this case,  $|\stab(\lambda_{5^{'}}^{3})|=6$ because $\lambda_{5^{'}}^{3} \in O_{5}^{3}$. The possibilities for $(x=a^{\varphi}, y=b^{\varphi}, z=c^{\varphi})$ are $(a,b,c)$, $(a,ab,c)$, $(b,a,c)$, $(b,ab,c)$, $(ab,a,c)$, $(ab,b,c)$. Clearly, $|\stab(\lambda_{5}^{4})|=6.8=48$.\\
We consider \\

{\footnotesize 
\begin{math}\begin{array}{llll}
\sigma_{1}=\left(\begin{array}{llll}
1 & 0 & 0 & 0 \\ 
1 & 1 & 0 & 0 \\ 
0 & 0 & 1 & 0 \\ 
0 & 0 & 0 & 1 \\ 
\end{array}\right)&\sigma_{2}=\left(\begin{array}{llll}
1 & 1 & 0 & 1 \\ 
0 & 1 & 0 & 0 \\ 
0 & 0 & 1 & 0 \\ 
0 & 0 & 0 & 1 \\ 
\end{array}\right)&\sigma_{3}=\left(\begin{array}{llll}
1 & 0 & 0 & 0 \\ 
1 & 1 & 0 & 1 \\ 
0 & 0 & 1 & 0 \\ 
0 & 0 & 0 & 1 \\ 
\end{array}\right)&\rho=\left(\begin{array}{llll}
1 & 0 & 0 & 0 \\ 
0 & 1 & 0 & 0 \\ 
0 & 0 & 1 & 1 \\ 
0 & 0 & 0 & 1 \\ 
\end{array}\right)
\end{array}\end{math}}

We have that $\rho$ and $\sigma_{i}$, $i=1,2,3$, are stabilizers of $\lambda_{5}^{4}$ of order 2. We still have  $\sigma_{1}\sigma_{3}=\sigma_{3}\sigma_{1}$, $\sigma_{1}\sigma_{2}\sigma_{1}=\sigma_{2}\sigma_{1}\sigma_{2}$, $\sigma_{2}\sigma_{3}\sigma_{2}=\sigma_{3}\sigma_{2}\sigma_{3}$ and $\rho \sigma = \sigma \rho$, for each $ \sigma \in S_{4}$. The group generated by $\sigma_{i}$, $i=1,2,3$ is isomorphic to $S_{4}$, $GR\{\rho\}\simeq \Z_{2}$. Therefore,
$\Out C_{5}^{4} \simeq S_{4}\times \Z_{2}$. 

\item Let $\lambda_{6}^{4}=(1111110100)$ and $\lambda_{1}^{3}=(111111)$. We have $|\stab(\lambda_{1}^{3})|=168$. In this case, only for $(\alpha_{1},\alpha_{2},\alpha_{3})=(0,0,0)$ we will have $\tilde{\varphi} \in \stab(\lambda_{6}^{4})$. In fact, $(a^{\varphi}d)^{2}=(b^{\varphi}d)^{2}=(c^{\varphi}d)^{2}=1$, since $(a^{\varphi})^{2}=(b^{\varphi})^{2}=(c^{\varphi})^{2}=-1$ and $d^{2}=-1$. Therefore, $|\stab(\lambda_{6}^{4})|=168$. We note that $a^{\varphi}, b^{\varphi}, c^{\varphi} \in \{a,b,c,ab,ac,bc,abc\}$. 

Let the following stabilizers of $\lambda_{6}^{4}$: $\sigma_{1}, \sigma_{2},\dots, \sigma_{6} $, represented in  matrix form by:\\

{
\footnotesize
\[\begin{array}{lll}
\sigma_{1}=\left(\begin{array}{llll}
1 & 1 & 0 & 0 \\ 
0 & 1 & 0 & 0 \\ 
0 & 0 & 1 & 0 \\ 
0 & 0 & 0 & 1 \\ 
\end{array}\right)\;\;&\;\; \sigma_{2}=\left(\begin{array}{llll}
1 & 0 & 1 & 0 \\ 
0 & 1 & 0 & 0 \\ 
0 & 0 & 1 & 0 \\ 
0 & 0 & 0 & 1 \\ 
\end{array}\right)\;\; &\;\; \sigma_{3}=\left(\begin{array}{llll}
1 & 0 & 0 & 0 \\ 
1 & 1 & 0 & 0 \\ 
0 & 0 & 1 & 0 \\ 
0 & 0 & 0 & 1 \\ 
\end{array}\right) \\ 
 
\end{array}\]}

{
\footnotesize
\[\begin{array}{lll}
\sigma_{4}=\left(\begin{array}{llll}
1 & 0 & 0 & 0 \\ 
0 & 1 & 1 & 0 \\ 
0 & 0 & 1 & 0 \\ 
0 & 0 & 0 & 1 \\ 
\end{array}\right)\;\;&\;\; \sigma_{5}=\left(\begin{array}{llll}
1 & 0 & 0 & 0 \\ 
0 & 1 & 0 & 0 \\ 
1 & 0 & 1 & 0 \\ 
0 & 0 & 0 & 1 \\ 
\end{array}\right)\;\; &\;\; \sigma_{6}=\left(\begin{array}{llll}
1 & 0 & 0 & 0 \\ 
0 & 1 & 0 & 0 \\ 
0 & 1 & 1 & 0 \\ 
0 & 0 & 0 & 1 \\ 
\end{array}\right) \\ 
 
\end{array}\]}

We have $a^{\sigma_{1}}=ab$, $b^{\sigma_{1}}=b$, $c^{\sigma_{1}}=c$ and $d^{\sigma_{1}}=d$. Hence $\sigma_{1} \in \stab(\lambda_{6}^{4})$. Analogously, we see that $\sigma_{2},\dots, \sigma_{6} \in \stab(\lambda_{6}^{4})$. Any product of the $\sigma_{i}'$s, $i=1,\dots,6$ is a stabilizer of $\lambda_{6}^{4}$. Hence $GR\{\sigma_{1},\dots,\sigma_{6}\}=\stab(\lambda_{6}^{4})$, but $GR\{\sigma_{1},\dots,\sigma_{6}\}\simeq GL_{3}(2)$, therefore,
$\stab(\lambda_{6}^{4})\simeq GL_{3}(2) $.

\item In this case, we have $\lambda_{7}^{4}=(0001000000)$, $\lambda_{2}^{3}=(000000)$ and $|\stab(\lambda_{2}^{3})|=24$. \linebreak Analogously to the previous case, we only have $\tilde{\varphi} \in \stab(\lambda_{7}^{4})$ for $(\alpha_{1},\alpha_{2},\alpha_{3})=(0,0,0)$. Therefore, $|\stab(\lambda_{7}^{4})|=24$.

Let $\sigma_{1}={\footnotesize \left(\begin{array}{llll}
1 & 0 & 0 & 0 \\ 
0 & 0 & 1 & 0 \\ 
0 & 1 & 0 & 0 \\ 
0 & 0 & 0 & 1 \\ 
\end{array}\right)}$, $\sigma_{2} = {\footnotesize \left(\begin{array}{llll}
0 & 1 & 0 & 0 \\ 
1 & 0 & 0 & 0 \\ 
0 & 0 & 1 & 0 \\ 
0 & 0 & 0 & 1 \\ 
\end{array}\right)}$, $\sigma_{3} = {\footnotesize \left(\begin{array}{llll}
1 & 0 & 0 & 0 \\ 
1 & 1 & 0 & 0 \\ 
1 & 0 & 1 & 0 \\ 
0 & 0 & 0 & 1 \\ 
\end{array}\right)}$. We have:

$
\sigma_{1}^{2}=\sigma_{2}^{2}=\sigma_{3}^{2}=1;$\; 
$\sigma_{1} \sigma_{2} \sigma_{1}=\sigma_{2} \sigma_{1} \sigma_{2};$\;
$\sigma_{3} \sigma_{2} \sigma_{3}=\sigma_{2} \sigma_{3} \sigma_{2}$ and
$\sigma_{1}\sigma_{3}=\sigma_{3}\sigma_{1}.\label{relC47-4}
$

 As $\sigma_{1}, \sigma_{2}, \sigma_{3} \in \stab(\lambda_{7}^{4}),$ then $\stab(\lambda_{7}^{4})=GR\{\sigma_{1}, \sigma_{2}, \sigma_{3}\;|\; \sigma_{1}^{2}=\sigma_{2}^{2}=\sigma_{3}^{2}=1, \sigma_{1} \sigma_{2} \sigma_{1}=\sigma_{2} \sigma_{1} \sigma_{2}, \sigma_{3} \sigma_{2} \sigma_{3}=\sigma_{2} \sigma_{3} \sigma_{2}, \sigma_{1}\sigma_{3}=\sigma_{3}\sigma_{1}\},$ but the group generated by $\sigma_{1}, \sigma_{2}, \sigma_{3}$ and which satisfies the relations above is isomorphic to the permutation group of $4$ elements, $S_{4}$. Therefore, $\stab(\lambda_{7}^{4})\simeq S_{4}$.

\item Let $\lambda_{8}^{4}=(0000001000)$ and  $\lambda_{2}^{3}=(000000)$. We suppose that  $\sigma \in \stab(\lambda_{8}^{4})$ and \linebreak $a^{\sigma}=u, b^{\sigma}=v, c^{\sigma}=w,d^{\sigma}=d$. Then 
\begin{eqnarray}
u^{2}=v^{2}=w^{2}=1;\;\;d^{2}=1\label{quadC48}\\
\left[u,v\right]=\left[u,w\right]=1\label{comC48-1}\\
\left[v,w\right]=\left[v,d\right]=\left[w,d\right]=1\label{comC48-2}\\
\left[u,d\right]=-1\label{comC48-3}
\end{eqnarray}
By (\ref{quadC48}) and (\ref{comC48-3}) we have $u \in \{a, ab, ac, abcd\}$ and by (\ref{quadC48}) and (\ref{comC48-2}) we have \linebreak $v,w \in \{b, c, bc, bd, cd, bcd\}$. Hence, by (\ref{comC48-1}) we obtain all the possible possibilities for $(u,v,w)$. For example, for $u=a$ we have $(u,v,w)= (a,b,c),$ $(a,b,bcd),$ $(a,c,b),$ $(a,c,bcd),$ $(a,bcd,b)$ or $(a,bcd,c)$. For any $u$ we obtain $6$ different basis. Therefore, in total we have $24$  different basis and thus, $|\stab(\lambda_{8}^{4})|=24$.

To prove that $\Out(C_{8}^{4})\simeq S_{4}$, we just need to consider the following stabilizers:

$\sigma_{1}={\footnotesize \left(\begin{array}{llll}
1 & 0 & 0 & 0 \\ 
0 & 1 & 0 & 0 \\ 
0 & 1 & 1 & 1 \\ 
0 & 0 & 0 & 1 \\ 
\end{array}\right)}$, $\sigma_{2} = {\footnotesize \left(\begin{array}{llll}
1 & 0 & 0 & 0 \\ 
0 & 1 & 1 & 1 \\ 
0 & 0 & 1 & 0 \\ 
0 & 0 & 0 & 1 \\ 
\end{array}\right)}$, $\sigma_{3} = {\footnotesize \left(\begin{array}{llll}
1 & 1 & 0 & 0 \\ 
0 & 1 & 0 & 0 \\ 
0 & 1 & 1 & 0 \\ 
0 & 0 & 0 & 1 \\ 
\end{array}\right)}$.

\item Let $\lambda_{9}^{4}=(0100001000)$ and $\sigma \in \stab(\lambda_{9}^{4})$ defined as above. In this case:
\begin{eqnarray}
u^{2}=w^{2}=d^{2}=1;\;\;v^{2}=-1 \label{quadC49}\\
\left[u,v\right]=\left[u,w\right]=1 \label{comC49-1}\\
\left[v,w\right]=\left[v,d\right]=\left[w,d\right]=1 \label{comC49-2}\\
\left[u,d\right]=-1 \label{comC49-3}
\end{eqnarray}
By (\ref{quadC49}) and (\ref{comC49-3}), we have $u \in \{a, ac, abc, abd\}$ and by (\ref{quadC49}) and (\ref{comC49-2}) we have \linebreak $v \in \{b, bc, bd, bcd\}$ and $w \in \{c, cd\}$. Now, by (\ref{comC49-1}) we have that $(u,v,w)$ are given by $(a,b,c)$, $(a,bcd,c)$, $(ac,bc,c)$, $(ac,bd,c)$, $(abc,bd,cd)$, $(abc,bcd,cd)$, $(abd,b,cd)$,\linebreak $(abd,bc,cd)$. Therefore, $|\stab(\lambda_{9}^{4})|=8$.

Let $\sigma={\footnotesize \left(\begin{array}{llll}
1 & 0 & 0 & 0 \\ 
0 & 1 & 1 & 1 \\ 
0 & 0 & 1 & 0 \\ 
0 & 0 & 0 & 1 \\ 
\end{array}\right)}$ and $\rho = {\footnotesize \left(\begin{array}{llll}
1 & 1 & 1 & 0 \\ 
0 & 1 & 1 & 1 \\ 
0 & 0 & 1 & 1 \\ 
0 & 0 & 0 & 1 \\ 
\end{array}\right)}$. With direct calculations we obtain that the orders of $\sigma$ and $\rho$ are $2$ and $4$, respectively, and the relation $\sigma \rho \sigma^{-1}=\rho^{-1}$ is valid, or better, $\rho \sigma = \sigma \rho^{3}$. We have that $\sigma$ and $\rho$ stabilizes $\lambda_{9}^{4}$.

 Hence $GR\{\sigma, \rho \;|\; \rho^{4}=1,\; \sigma^{2}=1,\;  \sigma \rho \sigma^{-1}=\rho^{-1}\} = \stab(\lambda_{9}^{4})$. Therefore, $\stab(\lambda_{9}^{4}) \simeq D_{8}.$


\item Let $\lambda_{10}^{4}=(0001111000)$ and  $\lambda_{5^{''}}^{3}=(000110)$. We have $|\stab(\lambda_{5^{''}}^{3})|=6$. In fact, if $\varphi \in \stab(\lambda_{5^{''}}^{3})$, then the possibilities for $\lambda_{5^{''}}^{3}$ are $(a^{\varphi}, b^{\varphi}, c^{\varphi})=(a,bc,c)$, $(a,c,bc)$, $(a,b,bc)$, $(a,bc,b)$, $(a,c,b)$ or $(a,b,c)$. 

Let $\tilde{\varphi}$ as before and we suppose that $(\alpha_{1},\alpha_{2},\alpha_{3})=(0,0,0)$. In this case, we have 6 different stabilizers of $\lambda_{10}^{4}$  of the form  $(a^{\varphi}, b^{\varphi}, c^{\varphi}, d)$, since $a^{\varphi}=a$, $b^{\varphi}=b,c$ or $bc$
and  $c^{\varphi}=b,c$ or $bc$. Now, assuming $(\alpha_{1},\alpha_{2},\alpha_{3})=(1,0,0)$, we have that $\tilde{\varphi}$ is  $(a^{\tilde{\varphi}},b^{\tilde{\varphi}},c^{\tilde{\varphi}},d^{\tilde{\varphi}})=(a^{\varphi}d, b^{\varphi}, c^{\varphi}, d)$ and hence, $(a^{\varphi}d)^{2}=(b^{\varphi})^{2}=(c^{\varphi})^{2}=1$, $[a^{\varphi}d,b^{\varphi}]=[a^{\varphi}d,c^{\varphi}]=[a^{\varphi}d,d]=-1$ and $[b^{\varphi},d]=[c^{\varphi},d]=[b^{\varphi},c^{\varphi}]=1$. Hence, we have more 6 stabilizers of $\lambda_{10}^{4}$.

For any other triple $(\alpha_{1},\alpha_{2},\alpha_{3})$ we will not have more stabilizers of $\lambda_{10}^{4}$, because if, for example, $(\alpha_{1},\alpha_{2},\alpha_{3})=(0,1,0)$ we will have $b^{\tilde{\varphi}}=b^{\varphi}d$, but $(b^{\varphi}d)^{2}=-1$ since $d^{2}=-1$ and $(b^{\varphi})^{2}=[b^{\varphi},d]=1$. The analysis is analogous for the other possibilities of $(\alpha_{1},\alpha_{2},\alpha_{3})$. Therefore, $|\stab(\lambda_{10}^{4})|=12$.

We consider the following stabilizers of $\lambda_{10}^{4}$:
{
\footnotesize
\[\begin{array}{lll}
\sigma_{1}=\left(\begin{array}{llll}
1 & 0 & 0 & 0 \\ 
0 & 1 & 0 & 0 \\ 
0 & 1 & 1 & 0 \\ 
0 & 0 & 0 & 1 \\ 
\end{array}\right)\;\;&\;\; \sigma_{2}=\left(\begin{array}{llll}
1 & 0 & 0 & 0 \\ 
0 & 1 & 1 & 0 \\ 
0 & 0 & 1 & 0 \\ 
0 & 0 & 0 & 1 \\ 
\end{array}\right)\;\; &\;\; \rho=\left(\begin{array}{llll}
1 & 0 & 0 & 1 \\ 
0 & 1 & 0 & 0 \\ 
0 & 0 & 1 & 0 \\ 
0 & 0 & 0 & 1 \\ 
\end{array}\right) \\ 
 
\end{array}\]}
We have $\sigma_{1}^{2}=\sigma_{2}^{2}=\rho^{2}=1$ and $\sigma \rho = \rho \sigma$, for each $ \sigma \in S_{3}.\Z_{2}$.

 Therefore, $\Out C_{10}^{4} \simeq S_{3}\times \Z_{2}$.

\item Let $\lambda_{11}^{4}=(0001001000)$. We have $u^{2}=v^{2}=w^{2}=1, d^{2}=-1$,\linebreak $[u,v]=[u,w]=1$, $[v,w]=[v,d]=[w,d]=1$ and $[u,d]=-1$. Thus, we obtain  $u\in \{a,ab,ac,ad,abd,acd\}$ and $v,w \in \{b,c,bc\}$. Hence, $|\stab(\lambda_{11}^{4})|=12$. 

We consider the following stabilizers of $\lambda_{11}^{4}$:

$\sigma={\footnotesize\left(\begin{array}{llll}
1 & 0 & 0 & 0 \\ 
0 & 0 & 1 & 0 \\ 
0 & 1 & 0 & 0 \\ 
0 & 0 & 0 & 1 \\ 
\end{array}\right)}$,  $\theta = {\footnotesize \left(\begin{array}{llll}
1 & 1 & 0 & 0 \\ 
0 & 1 & 0 & 0 \\ 
0 & 1 & 1 & 0 \\ 
0 & 0 & 0 & 1 \\ 
\end{array}\right)}$ e $\rho = {\footnotesize \left(\begin{array}{llll}
1 & 0 & 0 & 0 \\ 
0 & 1 & 0 & 0 \\ 
0 & 0 & 1 & 0 \\ 
0 & 0 & 0 & 1 \\ 
\end{array}\right)}$. 

 With  simple calculations we obtain $\sigma^{2}=\theta^{2}=\rho^{2}=1$, $\sigma \theta \sigma = \theta \sigma \theta$, $\sigma \theta \neq \theta \sigma$, $\sigma \rho=\rho \sigma$, $\theta \rho=\rho \theta$, $\sigma \theta \rho=\rho \sigma \theta$, $\sigma \theta \sigma \rho=\rho \sigma \theta \sigma$ e $\theta \sigma \rho=\rho \theta \sigma$. All these elements are stabilizers of $\lambda_{11}^{4}$. Let $H=GR\{\sigma, \theta \;|\; \sigma^{2}=\theta^{2}, \sigma \theta \sigma = \theta \sigma \theta\}$ and \linebreak $K=GR\{\rho\;|\; \rho^{2}=1\}$. Thus, from the before calculations, any stabilizer of $\lambda_{11}^{4}$ is written uniquely as the product of an element of $H$ by an element of $K$ and $xy=yx,$ for all $ x \in H, y \in K$. Besides, we know that $H \cong S_{3}$ and $K\cong \Z_{2}$. Therefore, we have  $\stab(\lambda_{11}^{4}) \simeq S_{3}\times \Z_{2}$.


\item Let $\lambda_{12}^{4}=(0000001100)$. We have $u^{2}=v^{2}=w^{2}=d^{2}=1$, $[u,v]=[u,w]=[v,d]=[w,d]=1$ and $[u,d]=[v,w]=-1$. Thus, we obtain $u\in \{a,ab,ac,abc\}$ and $v,w \in \{b,c\}$. Hence, $|\stab(\lambda_{12}^{4})|=8$.

In this case, we consider \;\;$\sigma={\footnotesize \left(\begin{array}{llll}
1 & 0 & 0 & 0 \\ 
0 & 0 & 1 & 0 \\ 
0 & 1 & 0 & 0 \\ 
0 & 0 & 0 & 1 \\ 
\end{array}\right)}$ \;\;and\;\; $\rho ={\footnotesize  \left(\begin{array}{llll}
1 & 1 & 0 & 0 \\ 
0 & 0 & 1 & 0 \\ 
0 & 1 & 0 & 0 \\ 
0 & 0 & 0 & 1 \\ 
\end{array}\right)}$, \\stabilizers of $\stab(\lambda_{12}^{4})$. Analogously the analysis of the case 9, we have \linebreak $GR\{\sigma, \rho \;|\; \rho^{4}=1,\; \sigma^{2}=1,\;  \sigma \rho \sigma^{-1}=\rho^{-1}\} = \stab(\lambda_{12}^{4})$. Therefore, $\stab(\lambda_{12}^{4}) \simeq ~D_{8}.$


\item Let $\lambda_{13}^{4}=(0110111100)$. In this case, $u^{2}=d^{2}=1$, $v^{2}=w^{2}=-1$, $[u,v]=[u,w]=[u,d]=[v,w]=-1$, $[v,d]=[w,d]=1$. Hence, we have $u \in \{a,ab,ac,abc\}$ and $v,w \in \{b,c,bc\}$. Thus, $|\stab(\lambda_{13}^{4})|=24$.

Here the required generators are given by:
{
\footnotesize
\[\begin{array}{lll}
\sigma_{1}=\left(\begin{array}{llll}
1 & 0 & 0 & 0 \\ 
0 & 1 & 0 & 0 \\ 
0 & 1 & 1 & 0 \\ 
0 & 0 & 0 & 1 \\ 
\end{array}\right)\;\;&\;\; \sigma_{2}=\left(\begin{array}{llll}
1 & 0 & 0 & 0 \\ 
0 & 0 & 1 & 0 \\ 
0 & 1 & 0 & 0 \\ 
0 & 0 & 0 & 1 \\ 
\end{array}\right)\;\; &\;\; \sigma_{3}=\left(\begin{array}{llll}
1 & 1 & 0 & 0 \\ 
0 & 1 & 0 & 0 \\ 
0 & 1 & 1 & 0 \\ 
0 & 0 & 0 & 1 \\ 
\end{array}\right) \\ 
 
\end{array}\]}
Analogously to the previous cases, we prove that $\Out(C_{13}^{4})\simeq S_{4}$.

\item For $\lambda_{14}^{4}=(0001001100)$ we have $u^{2}=v^{2}=w^{2}=1$, $d^{2}=-1$, $[u,v]=[u,w]=1$, $[v,d]=[w,d]=1$ and $[u,d]=[v,w]=-1$. Hence, $u \in \{a,ab,ac,ad,abc,abd,acd,abcd\}$ and $v,w \in \{b,c,bcd\}.$ Therefore, $|\stab(\lambda_{14}^{4})|=48$. 

We consider the following generators for $\Z_{2}^{3}$:
{
\footnotesize
\[\begin{array}{lll}
\rho_{1}=\left(\begin{array}{llll}
1 & 1 & 0 & 0 \\ 
0 & 1 & 0 & 0 \\ 
0 & 0 & 1 & 0 \\ 
0 & 0 & 0 & 1 \\ 
\end{array}\right)\;\;&\;\; \rho_{2}=\left(\begin{array}{llll}
1 & 0 & 0 & 1 \\ 
0 & 1 & 0 & 0 \\ 
0 & 0 & 1 & 0 \\ 
0 & 0 & 0 & 1 \\ 
\end{array}\right)\;\; &\;\; \rho_{3}=\left(\begin{array}{llll}
1 & 0 & 1 & 0 \\ 
0 & 1 & 0 & 0 \\ 
0 & 0 & 1 & 0 \\ 
0 & 0 & 0 & 1 \\ 
\end{array}\right) \\ 
 
\end{array}\]}
For $S_{3}$, we consider:{
\footnotesize
$\begin{array}{ll}
\sigma_{1}=\left(\begin{array}{llll}
1 & 0 & 0 & 0 \\ 
0 & 0 & 1 & 0 \\ 
0 & 1 & 0 & 0 \\ 
0 & 0 & 0 & 1 \\ 
\end{array}\right)\;\;&\;\; \sigma_{2}=\left(\begin{array}{llll}
1 & 1 & 0 & 1 \\ 
0 & 1 & 0 & 0 \\ 
0 & 1 & 1 & 1 \\ 
0 & 0 & 0 & 1 \\ 
\end{array}\right) \\ 
 
\end{array}$}

With the required calculations we obtain $\Out C_{14}^{4} \simeq \Z_{2}^{3}\rtimes S_{3}$.

\item For $\lambda_{15}^{4}=(1001001100)$ we have $u^{2}=-1$, $v^{2}=w^{2}=1$, $d^{2}=-1$, $[u,v]=[u,w]=[v,d]=[w,d]=1$ and $[u,d]=[v,w]=-1$. Thus, $u \in \{a,ab,ac,ad,abc,abd,acd,abcd\}$ and $v,w \in \{b,c,bcd\}.$ Therefore, $|\stab(\lambda_{15}^{4})|=48$. 

We consider the following generators for $\Z_{2}^{3}$:
{
\footnotesize
\[\begin{array}{lll}
\rho_{1}=\left(\begin{array}{llll}
1 & 1 & 0 & 0 \\ 
0 & 1 & 0 & 0 \\ 
0 & 0 & 1 & 0 \\ 
0 & 0 & 0 & 1 \\ 
\end{array}\right)\;\;&\;\; \rho_{2}=\left(\begin{array}{llll}
1 & 0 & 0 & 1 \\ 
0 & 1 & 0 & 0 \\ 
0 & 0 & 1 & 0 \\ 
0 & 0 & 0 & 1 \\ 
\end{array}\right)\;\; &\;\; \rho_{3}=\left(\begin{array}{llll}
1 & 0 & 1 & 0 \\ 
0 & 1 & 0 & 0 \\ 
0 & 0 & 1 & 0 \\ 
0 & 0 & 0 & 1 \\ 
\end{array}\right) \\ 
 
\end{array}\]}
For $S_{3}$, we consider:{
\footnotesize
$\begin{array}{ll}
\sigma_{1}=\left(\begin{array}{llll}
1 & 0 & 0 & 0 \\ 
0 & 0 & 1 & 0 \\ 
0 & 1 & 0 & 0 \\ 
0 & 0 & 0 & 1 \\ 
\end{array}\right)\;\;&\;\; \sigma_{2}=\left(\begin{array}{llll}
1 & 1 & 0 & 1 \\ 
0 & 1 & 0 & 0 \\ 
0 & 1 & 1 & 1 \\ 
0 & 0 & 0 & 1 \\ 
\end{array}\right) \\ 
 
\end{array}$}

In this case, we also obtain $\Out C_{15}^{4} \simeq \Z_{2}^{3}\rtimes S_{3}$.

\item Let $\lambda_{16}^{4}=(0001111100)$, then $u^{2}=v^{2}=w^{2}=1$, $d^{2}=-1$, $[u,v]=[u,w]=-1$, $[u,d]=[v,w]=-1$ and $[v,d]=[w,d]=1$, and hence, $u \in \{a,abc,ad,abcd\}$ and $v,w \in \{b,c\}$. Therefore, $|\stab(\lambda_{16}^{4})|=8$. 

Let the following stabilizers of $\lambda_{16}^{4}$:

$\sigma={\footnotesize\left(\begin{array}{llll}
1 & 0 & 0 & 0 \\ 
0 & 0 & 1 & 0 \\ 
0 & 1 & 0 & 0 \\ 
0 & 0 & 0 & 1 \\ 
\end{array}\right)}$,  $\theta = {\footnotesize \left(\begin{array}{llll}
1 & 0 & 0 & 1 \\ 
0 & 1 & 0 & 0 \\ 
0 & 0 & 1 & 0 \\ 
0 & 0 & 0 & 1 \\ 
\end{array}\right)}$ and $\rho = {\footnotesize \left(\begin{array}{llll}
1 & 1 & 1 & 0 \\ 
0 & 0 & 1 & 0 \\ 
0 & 1 & 0 & 0 \\ 
0 & 0 & 0 & 1 \\ 
\end{array}\right)}$. 

We have $\sigma^{2}=\theta^{2}=\rho^{2}=1$ and $\sigma \theta=\theta \sigma$, $\sigma \rho=\rho \sigma$ and $\sigma \rho=\rho \sigma$.  Clearly, $\stab(\lambda_{16}^{4})\simeq \Z_{2}^{3}$.


\end{enumerate}

\end{proof}

\section{Representations of Code Loops}

First of all, we defined code loops using double even codes fixed. Now, we fix a code loop $L$ and we want to determine the double even codes $V$ such that $L\simeq L(V)$.

A representation of a given code loop $L$ is a double even code $V \subseteq {{\bf F}_{2}^{m}}$ such that $L \simeq L(V)$. The degree of a representation is the number $m$.

We notice that there are many different representations for a same code loop. We prove by the Theorems \ref{theorem4.4} and \ref{theorem4.6} below that, there are representations of nonassociatives code loops of rank $3$ and $4$ such that the degree of each representation is the smallest possible.

\begin{definition}
A representation $V$ is called basic if the degree of $V$ is minimal.
\end{definition}

We identify the ${\bf F}_{2}-$space ${{\bf F}_{2}^{m}}$ as the set of all subsets of $I_{m}=\left\{1,\dots,m\right\}$ and we define a relation of  equivalence $\sim$ on $I_{m}$:
 $i \sim j$ if and only if  $\left\{i,j\right\} \cap v = \left\{i,j\right\}$ or $\left\{i,j\right\} \cap v = \emptyset$, for all $ v \in V$.

 We notice that this definition  is equivalent to: $i \sim j$ if and only if $\left\{i,j\right\} \cap ~ v_{k} = \left\{i,j\right\}$ or $\left\{i,j\right\} \cap ~ v_{k} = \emptyset$, $k=1,\dots,s$ and $\left\{v_{1},\dots,v_{s}\right\}$ is a basis of $V$.
 
 We will consider only representations such that, for any equivalence classes $X$, we have $|X| < 8.$ We call this representations by reduced representations.

Our main problem is to find all basic representations for a given code loop.

\begin{definition}
For a given representation $V$ and all the equivalence classes $X_{1},...,X_{r},$  the {\bf type} of $V$ is a vector $(|X_{1}|,...,|X_{r}|)$ such that $|X_{1}| \leq |X_{2}| \leq ...\leq |X_{r}|$.
\end{definition} 

\begin{definition}
Let $V_{1}$ and $V_{2}$ be double even codes of ${\bf F}_{2}^{m}$. We say that $V_{1}$ and $V_{2}$ are  isomorphic even codes if and only if there is a bijection $\varphi \in S_{m}$ such that $V_{1}^{\varphi}=V_{2}$.
\end{definition}

\begin{theorem}\label{theorem4.4}
The code loops $C_{1}^{3},\dots,C_{5}^{3}$ have the following basic representations $V_{1},\dots,V_{5}$, which are given by\\
$V_{1}=\left\langle (1234),(1256),(1357)\right\rangle,$\\
$V_{2}=\left\langle (12345678),(12349,10,11,12),(15679,10,11,13)\right\rangle,$\\
$V_{3}=\left\langle (12345678),(1234569,10),(1234579,11)\right\rangle,$\\
$V_{4}=\left\langle (1234),(1256789,10,11,12,13,14),(1356789,10,11,15,16,17)\right\rangle,$\\
$V_{5}=\left\langle (12345678910,11,12),(1-8,13,14,15,16),(1-5,9,10,11,13,14,15,17)\right\rangle.$
\end{theorem}

\begin{proof}
We consider the ${\bf F}_{2}$-subspaces $V_{1},\dots,V_{5}$ of ${{\bf F}_{2}^{7}},{{\bf F}_{2}^{13}},{{\bf F}_{2}^{11}},{{\bf F}_{2}^{17}},{{\bf F}_{2}^{17}},$ respectively, as above. First we will see that each space $V_{i}$  is a representation of $C_{i}^{3}$, that is, is a double even code of ${{\bf F}_{2}^{m}}$, for some $m$ and that $C_{i}^{3} \simeq L(V_{i})$, $i=1,\dots,5$. 

Since the elements of $V_{5}$ are $v_{0}=0$, $v_{1}=(12345678910,11,12)$, \linebreak  $v_{2}=(12345678,13,14,15,16)$, $v_{3}=(12345,9,10,11,13,14,15,17)$, $v_{4}=v_{1}+v_{2}=(9-16)$, $v_{5}=v_{1}+v_{3}=(678,12,13,14,15,17)$, $v_{6}=v_{2}+v_{3}=(6789,10,11,16,17)$ and $v_{7}=v_{1}+v_{2}+v_{3}=(12345,12,16,17)$, we see, clearly, that all the vectors have weight with multiplicity  $4$ and the weight of the intersection of each pair of vectors is even. Thus, $V_{5}$ is a double even code. Analogously, we prove that $V_{i}$, $i=1,\dots,4$ is a double even code.

Now, the isomorphism $C_{i}^{3}\simeq L(V_{i})$ follows directly from Theorem \ref{theorem3.1} of Classification of Code Loops of rank $3$ and from Proposition \ref{prop3.2}. We just need to calculate the characteristic vector associated to $L(V_{i})$ and note that it belongs to the orbit $O_{i}^{3}$ corresponding to the code loop $C_{i}^{3}$.
 
As example, we calculate the characteristic vector associate to $L(V_{5})=\{1,-1\}\times V_{5}$. We have $ v_{i}^{2}=(-1)^{\frac{|v_{i}|}{4}}=-1$ and $ \left[v_{i},v_{j}\right]=(-1)^{\frac{|v_{i}\cap v_{j}|}{2}}=1,$ for $i,j=1,2,3$. Hence $\lambda(L(V_{5}))=(111000) \in O_{5}^{3}$. Therefore, $C_{5}^{3}\simeq L(V_{5}).$

We will demonstrate now that each $V_{i}$, $i=1,\dots,5$, up to isomorphism, is the unique basic representation of the code loop $C_{i}^{3}$. We consider $X=(a,b,c)$ a set of generators for $C_{i}^{3}$ such that $\lambda = \lambda_{X}(C_{i}^{3})$ is the corresponding characteristic vector, and we assume that $V$ is a basic representation of $C_{i}^{3}$, where $v,w,u$ are the elements of the basis of $V$ which corresponds to $a,b,c$. We use the notation $t=|v \cap w\cap u|$. We remember that $t \equiv 1(mod\;2)$, because $(a,b,c)=-1$.

{\itshape Case} $i=1$: In this case, we are assuming $gr \;V \leq 7$. The characteristic vector is $\lambda = (111111)$, then $|v|\equiv |w|\equiv |u|\equiv 4 (mod8)$. We suppose that $v=(1234)$, but $[a,b]=-1$, then $|v \cap w| \equiv 2 (mod 4)$, and hence, $|v\cap w|=2$. Analogously, we obtain $|v\cap u|=2$. Let $w=(1256)$, hence we also obtain $|w\cap u|=2$ and then, $t=1$. Hence, $u=(1357)$. Therefore, $V=V_{1}$.

{\itshape Case} $i=2$: The characteristic vector is $\lambda = (000000)$, then $|v|\equiv |w|\equiv |u|\equiv 0 (mod8)$. Since $gr\; V \leq 13$, then $|v|=8$. Analogously, $|w|=|u|=8$. Let $v=(12345678)$. As $|v \cap w| \equiv 0 (mod 4)$, hence $|v\cap w|=4$. Analogously, $|v\cap u|=|w\cap u|=4$.

Let $w=(12349,10,11,12)$. We have two possibilities for $t$: $t=1$ or $t=~3$. Case $t=3$, we will have $gr\; V = 15$, a contradiction. Hence, $t=1$ and $u=(15679,10,11,13)$. Therefore, $V=V_{2}$.

{\itshape Case} $i=3$: We consider $\lambda = (000111)$ and $gr\; V \leq 11$. Analogously to the previous case, we have $|v|=|w|=|u|=8$. In this case, we have $|v\cap w|=|v\cap u|=|w\cap u|=6$. We suppose $v=(12345678)$ and $w=(1234569,10)$. We have three possibilities for $t$: $t=1,3$ or $5$. Case $t\leq 3$ we will have $|v|> 8$, a contradiction. Hence, $t=5$ and we can assume $u=(1234579,11)$. Therefore, $V=V_{3}$.

{\itshape Case} $i=4$: In this case, $\lambda = (111110)$. Since $|v|\equiv |w|\equiv |u|\equiv 4 (mod8)$ and $|w\cap u|\equiv 0 (mod 4)$ then $|w|=|u|=12$ and $|w\cap u|=8$. Let $v=(1234)$, so $|v\cap w|=2$. If $|v\cap w|=4$, then $[v,w]=(-1)^{\frac{|v\cap w|}{2}}=1,$ which does not occur since $[a,b]=-1$.

Let $w=(1256789,10,11,12,13,14)$.  Since $|v\cap w|=2,$ then $t=1$. Hence\linebreak $u=(1356789,10,11,15,16,17)$. Therefore, $V = V_{4}$. 

Now, we suppose $v=(123456789,10,11,12)$. If $|v\cap w| \leq 6$ then $gr\;V \geq 18.$ In fact, we have $|v\cap w|=2,6$ or $10$. If $|v\cap w|=2$ we will have $|v+w|=20$, a contradiction. If $|v\cap w|=6$ we will have $|v+w|=12$ and hence, $gr\; V \;\geq 18$, a contradiction. Therefore, $|v\cap w| =10$. Analogously we obtain $|v\cap u|=10$. Without loss of generality, we suppose $w=(1256789,10,11,12,13,14)$. Since $t\equiv 1(mod 2)$, then the possibilities for $t$ are $1,3,4$ or $7$. In any case we will have $|v| \geq 13,$ which is a contradiction. Then, there is not this last possibility for $v$.

{\itshape Case} $i=5$: The characteristic vector is $\lambda = (111000)$, so we have $|v|\equiv |w|\equiv |u|\equiv 4 (mod 8)$.  Assuming $v=(1234)$ and since $[a,b]=1$, we will have $|v \cap w| \equiv 0 (mod 4)$. Case $v\cap w \neq \emptyset$ we will have $|v\cap w|=4$, and hence $v \subset w$, which is a contradiction, because this give us $t \equiv 0 (mod\;2)$, which does not occur. The case $v\cap w = \emptyset$ also does not occur, since we must have nonempty intersection between $v, w$ and $u$.

Therefore $|v|\geq 12$. Analogously, we prove that $|w|\geq 12$ and $|u|\geq 12$.

The representation $V$ is basic and $gr\; V \leq 17$, so we obtain $|v|=|w|=|u|=12$. Without loss of generality, we suppose $v=(123456789,10,11,12)$ and $|w\cap v| = 4$ or $8$.  If $|v\cap w|=4$, then $|v+w|=16$ and hence, $gr \;V \geq |v\cap w|+|v+w|=20,$ which is a contradiction.

Therefore, $v\cap w = (12345678)$ and $w=(12345678,13,14,15,16).$ Analogously, we have $|v\cap u|=|w\cap u|=8.$

If $t\leq 3$, then $|v|\geq |v\cap w\cap u|+|(v\cap w)\backslash
(v\cap w\cap u)|+|(v\cap u)\backslash
(v\cap w\cap u)|\geq 3+5+5=13$, a contradiction. Therefore, $t\geq 5.$

If $t=7$, then $gr\; V \geq 19$. To prove this, we consider $u=(i_{1},\dots,i_{12})$, where \linebreak $i_{1},\dots,i_{7} \in v\cap w\cap u$, but since $v\cap w = (12345678)$ then $i_{1},\dots,i_{7} \in v\cap w$. Suppose $u=(1234567,i_{8},\dots,i_{12})$. Since $|u\cap v|=8$ and $|u\cap w|=8$, so $i_{8} \in \left\{9,10,11,12\right\}$ whereas $i_{9} \in \left\{13,14,15,16\right\}$. We choose $i_{8}=10$ and $i_{9}=13.$ Hence $i_{j} \notin \left\{9,11,12,14,15,16\right\}$, $j=10,11,12$. Hence, a possibility for $u$ it will be $u=(1234567,10,13,17,18,19)$ so that $gr\; V \geq 19$, a contradiction with the fact $V$ be basic representation. Thus, $t=5$ and  $u=~(123459,10,11,13,14,15,17)$. Therefore, $V=V_{5}$.

\end{proof}

\begin{corollary}
Each basic representation of the code loops $C_{1}^{3},\dots,C_{5}^{3}$ has the following types, respectively:
$(1111111),(1111333),(1111115),(1111337),(1113335).$
 \end{corollary}

According with the Theorem \ref{theorem3.5}, we have exactly $16$ code loops of rank $4$, namely,  $C_{1}^{4},C_{2}^{4},\dots,$ \linebreak$C_{16}^{4}.$ For each $C_{i}^{4}$, $i=1,\dots,16,$ we have to find $V_{i} \subseteq {{\bf F}_{2}^{m}}$ double even code of minimal degree $m$ such that $V_{i} \cong L(C_{i}^{4}).$
 
In general, the set $X=\left\{a,b,c,d\right\}$ represents a set of generates of $C_{i}^{4}$ such that $\lambda_{X}(C_{i}^{4})=\lambda(C_{i}^{4})$ is its corresponding characteristic vector. We also suppose that  $V_{i}=ger\left\{v_{1},v_{2},v_{3},v_{4}\right\}$ is a basic representation of $C_{i}^{4}$, where $v_{1},v_{2},v_{3},v_{4}$ corresponds to $a,b,c,d$ respectively.
 
 For the next theorem we use the notation: $t_{ijk}=|v_{i}\cap v_{j}\cap v_{k}|, i,j,k = 1,..,4$ and $t_{1234}=|v_{1}\cap v_{2}\cap v_{3}\cap v_{4}|.$
 
\begin{theorem}\label{theorem4.6}
Each code loop $C_{1}^{4},\dots,C_{16}^{4}$ has the following set of generators to its basic representation $V_{1},\dots,V_{16}$, respectively:\\
$V_{1}=\left\langle (1234),(1256),(1357),(1-8)\right\rangle,$\\
$V_{2}=\left\langle (1-8),(1-4,9-12),(15679,10,11,13),(12589,12,13,14)\right\rangle,$\\
$V_{3}=\left\langle (1-8),(1-6,9,10),(1-5,7,9,11),(1,6-12)\right\rangle,$\\
$V_{4}=\left\langle (1-8),(1-6,9,10),(12379,11-17),(1478,9-11,18)\right\rangle,$\\
$V_{5}=\left\langle (1-8),(1234,9-12),(159,13-17),(12569,10,13,18)\right\rangle,$\\
$V_{6}=\left\langle (1234),(1256),(1357),(8,9,10,11)\right\rangle,$\\
$V_{7}=\left\langle (1-8),(1234,9-12),(15679,10,11,13),(14,15,16,17)\right\rangle,$\\
$V_{8}=\left\langle (1-8),(1234,9-12),(12359,13,14,15),(1,2,10,11,13,14,16,17)\right\rangle,$\\
$V_{9}=\left\langle (1-8),(1234,9-16),(15679,10,11,17),(5,6,9,10,12,13,18,19)\right\rangle,$\\
$V_{10}=\left\langle (1-8),(1,2,9-14),(139,10,11,15-17),(4,5,18,19)\right\rangle,$\\
$V_{11}=\left\langle (1-8),(1-4,9-12),(12359,13-15),(6,7,16,17)\right\rangle,$\\
$V_{12}=\left\langle (1-8),(1-4,9-12),(1235,9-11,13),(129,10,14-17)\right\rangle,$\\
$V_{13}=\left\langle (1-8),(129,10),(139,11),(45,12-17)\right\rangle,$\\
$V_{14}=\left\langle (1-8),(1-4,9-12),(15679,13-15),(2,3,10,11)\right\rangle,$\\
$V_{15}=\left\langle (1-12),(1-4,13-16),(1235,13-17),(1-6,13-18)\right\rangle,$\\
$V_{16}=\left\langle (1-8),(1,2,9-14),(1,3-7,9-13,15-19),(1-5,8,9,14-17,20)\right\rangle.$
\end{theorem}

\begin{proof}
Analogously to the Theorem  \ref{theorem4.4} (Representations of the code loops of rank 3), we demonstrate that each $V_{i}$, $i=1,\dots,16$, is a double even code. Now, to prove that $V_{i}\simeq L(C_{i}^{4})$ we just need to find the characteristic vector associated to $L(C_{i}^{4})$ and apply the Theorem \ref{theorem3.5} (Classification of code loop of rank 4). Therefore, $V_{i}$ is a representation of  $C_{i}^{4}$, $i=1,\dots,16$.

Now, we are going to prove, up to isomorphism, that $V_{1}$ is the unique basic representation of $C_{1}^{4}$. We consider $(a,b,c,d)$ a set of generators of $C_{1}^{4}$ such that $\lambda=\lambda(C_{1}^{4})=(1110110100)$. We suppose that $V=ger\{v_{1},v_{2},v_{3},v_{4}\}$ is a basic representation of $C_{1}^{4}$, where $v_{1},v_{2},v_{3},v_{4}$ corresponds to $a,b,c,d$, respectively. Hence, $gr \;V \leq 8$.

In this case, we have 
\begin{eqnarray*}
|v_{1}|&\equiv & |v_{2}|\equiv |v_{3}|\equiv 4(mod\;8)\\
|v_{4}|&\equiv & 0(mod\;8)\\
|v_{1}\cap v_{2}|&\equiv & |v_{1}\cap v_{3}|\equiv |v_{2}\cap v_{3}|\equiv 2 (mod\;4)\\
|v_{1}\cap v_{4}|&\equiv & |v_{2}\cap v_{4}|\equiv |v_{3}\cap v_{4}|\equiv 0 (mod\;4)
\end{eqnarray*}

Suppose $v_{1}=(1234)$, then $|v_{1}\cap v_{2}| = |v_{1}\cap v_{3}|=2$ and hence, $t_{123}=1$. Then we can assume $v_{2}=(1256)$ and $v_{3}=(1357)$.

We will analyse two possible cases for values of $t_{1234}$: $0$ and $1$.

We will write whenever necessary $t_{ij4}$, for $|v_{i}\cap v_{j}\cap v_{4}|$ with  $i,j=1,2,3$, $i\neq j$ and $t_{1234}$ for $|v_{1}\cap v_{2}\cap v_{3}\cap v_{4}|$.

Case $t_{1234}=0$, we have $t_{ij4}=0$ and thus, $|v_{i}\cap v_{4}|=0$.  Therefore, $|v_{4}|\geq 8$ and then, we don't have basic reduced representation in this case. Case $t_{1234}=1$, we have $t_{ij4}=2$ and thus, $|v_{i}\cap v_{4}|=4$, that is, $v_{i}\subset v_{4}$, $i=1,2,3$. Next, $v_{4}=(12345678)$. Therefore, $V=V_{1}$. 

If $v_{1}=(123456789,10,11,12)$ we have $gr\; V > 12$, which contradicts the hypothesis of $V$ to be basic.

Now, we will prove that $V_{7}$, up to isomorphism, is the unique basic representation of $C_{7}^{4}$. Here the characteristic vector is given by $\lambda=(0001000000)$. We suppose that $V=ger\{v_{1},v_{2},v_{3},v_{4}\}$ is a basic representation of $C_{7}^{4}$. Then:
\begin{eqnarray*}
|v_{1}|&\equiv & |v_{2}|\equiv |v_{3}|\equiv 0(mod\;8)\\
|v_{4}|&\equiv & 4(mod\;8)\\
|v_{1}\cap v_{2}|&\equiv & |v_{1}\cap v_{3}|\equiv |v_{2}\cap v_{3}|\equiv 0 (mod\;4)\\
|v_{1}\cap v_{4}|&\equiv & |v_{2}\cap v_{4}|\equiv |v_{3}\cap v_{4}|\equiv 0 (mod\;4)
\end{eqnarray*}

Let $v_{1}=(12345678)$, so $|v_{1}\cap v_{2}|=|v_{1}\cap v_{3}|=4$ and hence, $t_{123}=1$ or $3$. Suppose $v_{2}=(12349,10,11,12)$, so $|v_{2}\cap v_{3}|=4$. Case $t_{123}=1$, we consider $v_{3}=(1567,9,10,11,13)$. If $t_{1234}=0$, then $t_{ij4}=0$ or $2$. Considering that $|v_{i}\cap v_{j}|\equiv 0 (mod\;4)$, $i,j=1,2,3$, $i\neq j$, then we have only two subcases for analyze:
\begin{itemize}
	\item $t_{ij4}=0$: In this subcase, $|v_{i}\cap v_{4}|=0$, for $ \; i=1,2,3$ and hence, we can assume \\$v_{4}=(14,15,16,17)$. Thus, for this case, $V=V_{7}$. 
	\item $t_{ij4}=2$: Here, $|v_{i}\cap v_{4}|=4$, for $ \; i=1,2,3$ and hence,  $v_{4}=(2,3,5,6,9,10,14-~19)$, which contradicts the minimal degree of $V$.
\end{itemize}

If $t_{1234}=1$, then $t_{ij4}=2$ or $4$. Analogously, we have two subcases for analyze:
\begin{itemize}
	\item $t_{ij4}=2$: In this subcase, $|v_{i}\cap v_{4}|=4$, for $ \; i=1,2,3$, which give us $gr\; V > 17$, a contradiction.
	\item $t_{ij4}=4$: In this subcase, $|v_{i}\cap v_{4}|=8$, for $ \; i=1,2,3$, which also contradicts the minimal degree of $V$.
\end{itemize}

Now, analyzing the case $t_{123}=3$, we suppose that $v_{3}=(12359,13,14,15)$. If $t_{1234}=0,$ then $|v_{i}\cap v_{4}|=0$ and hence, we will have $v_{4}=(16,17,18,19)$, which contradicts the degree of $V$ to be minimal. Analogously, for the cases $t_{1234}=1,2$ and $4$, we will have $gr\; V > 17$. 

Now, let the code loop $C_{10}^{4}$ with $\lambda=(0001111000)$ and $V=ger\{v_{1},v_{2},v_{3},v_{4}\}$ its basic representation. Let $v_{1}=(12345678)$, then $|v_{1}\cap v_{2}|=2$ or $6$ and $|v_{1}\cap v_{3}|=2$ or $6$. Case $|v_{1}\cap v_{2}|=2$, we can assume $v_{2}=(1,2,9-14)$. Hence, we have $t_{123}=1$ and $|v_{2}\cap v_{3}|=4$. 
\begin{itemize}
\item For $|v_{1}\cap v_{3}|=2$, consider $v_{3}=(1,3,9,10,11,15,16,17)$. If $t_{1234}=0$, then $t_{124}=t_{134}=0$ and $t_{234}=0$ or $2$. Thus $|v_{1}\cap v_{4}|=2$. If $t_{234}=0$: $|v_{2}\cap v_{4}|=0$ and $|v_{3}\cap v_{4}|=0$. Thus, we have $v_{4}=(4,5,18,19)$ and, therefore, $V=V_{10}$. In the case $t_{234}=2$ we will find $v_{4}=(4,5,9,10,12,13,15,16,18-21)$, contradicting the minimality of the degree of $V$. The analyze of $t_{1234}=1$ is analogous.

\item For $|v_{1}\cap v_{3}|=6$, we can consider $v_{3}=(1,3-7,9-11,15-21)$. In this case $gr\; V > 19$ for any analyze.

\end{itemize}

We don't have basic representation in case $|v_{1}\cap v_{2}|=6$. 

Analogously, in the other cases, we prove that each $V_{i}$ is the unique basic representation, up to isomorphism.

\end{proof}

\newpage
\begin{corollary}Each basic representation of the code loops $C_{1}^{4},\dots,C_{16}^{4}$ has the following degree and type, respectively:
\vspace{3mm}

\begin{center}{
\begin{tabular}{l|l|l|l|l|l}
\multicolumn{1}{c|}{$i$} & \multicolumn{1}{c|}{$deg\;V_{i}$} & \multicolumn{1}{c|}{type of $V_{i}$}&\multicolumn{1}{c|}{$i$} & \multicolumn{1}{c|}{$deg\;V_{i}$} & \multicolumn{1}{c}{type of $V_{i}$} \\ 
\hline
\multicolumn{1}{c|}{1} & \multicolumn{1}{c|}{8} & (11111111) & 
\multicolumn{1}{c|}{2} & \multicolumn{1}{c|}{14} & (11111111222) \\ 
\multicolumn{1}{c|}{3} & \multicolumn{1}{c|}{12} & (111111114) &
\multicolumn{1}{c|}{4} & \multicolumn{1}{c|}{18} & (11111111226) \\ 
\multicolumn{1}{c|}{5} & \multicolumn{1}{c|}{18} & (111111112224) &
\multicolumn{1}{c|}{6} & \multicolumn{1}{c|}{11} & (11111114) \\ 
\multicolumn{1}{c|}{7} & \multicolumn{1}{c|}{17} & (11113334) &
\multicolumn{1}{c|}{8} & \multicolumn{1}{c|}{17} & (11111122223) \\ 
\multicolumn{1}{c|}{9} & \multicolumn{1}{c|}{19} & (11111222233) &
\multicolumn{1}{c|}{10} & \multicolumn{1}{c|}{19} & (111223333) \\ 
\multicolumn{1}{c|}{11} & \multicolumn{1}{c|}{17} & (111122333) & 
\multicolumn{1}{c|}{12} & \multicolumn{1}{c|}{17} & (1111112234) \\ 
\multicolumn{1}{c|}{13} & \multicolumn{1}{c|}{17} & (111111236) &
\multicolumn{1}{c|}{14} & \multicolumn{1}{c|}{15} & (111112233) \\ 
\multicolumn{1}{c|}{15} & \multicolumn{1}{c|}{18} & (111111336) &
\multicolumn{1}{c|}{16} & \multicolumn{1}{c|}{20} & (11111122334) \\ 
\end{tabular}}
\end{center}
\end{corollary}

Note that in the case of code loops of rank 3 and 4 the type of code loop define this loop up to isomophism. May be it is true in general case.

\begin{conjecture}
Let $V_{1}$ and $V_{2}$ be representations of a code loop $L$. If this representations have the same degree and type, then $V_{1}$ and $V_{2}$ are isomorphic. 
\end{conjecture}

%


\begin{thebibliography}{0}

\bibitem{1} \textsc{R. M. Pires}, Loops de c\'{o}digo: automorfismos e representa\c c\~{o}es. 2011. 93f. Tese (Doutorado em Matem\' atica)-Instituto de Matem\'atica e Estat\'istica, Universidade de S\~ao Paulo, S\~ao Paulo, 2011.

\bibitem{2} R.L. Griess Jr., {\em Code loops}, J.Algebra 100 (1986),224$-$234.

\bibitem{3}R.H. Bruck, {\it A survey of binary systems}, Springer-Verlag (1958).

\bibitem{4}O. Chein and E.G. Goodaire, {\it Moufang Loops with a Unique Nonidentity Commutator (Associator, Square)}, J. Algebra 130 (1990), 369-384.

\bibitem{5}H.O. Pflugfelder, {\it Quasigroups and Loops: An Introduction}, Berkin:Heldermann (1990).

\bibitem{6}O. Chein, H.O. Pflugfelder, and J.D.H. Smith, {\it Quasigroups and Loops: Theory and Applications}, Berkin:Heldermann (1990).

\bibitem{7}E.G. Goodaire, E. Jaspers and C. Polcino, {\it Alternative Loop Rings}, North Holland Math, Studies N.184, Elsevier, Amsterdam (1996).

\end{thebibliography}
\end{document}